\documentclass[11pt]{article}
\usepackage[english]{babel}
\usepackage[utf8]{inputenc}
\usepackage{amsmath,amsfonts,amssymb,amsthm,mathrsfs}
\usepackage{minitoc}

\renewcommand{\leq}{\leqslant}
\renewcommand{\geq}{\geqslant}
\usepackage[cmintegrals,libertine]{newtxmath}
\usepackage[cal=euler]{mathalfa}
\usepackage[mono=false]{libertine}
\useosf

\usepackage[dvipsnames]{xcolor}
\usepackage[a4paper,vmargin={3.5cm,3.5cm},hmargin={2.5cm,2.5cm}]{geometry}
\usepackage[font=sf, labelfont={sf,bf}, margin=1cm]{caption}
\usepackage{graphicx,graphics}
\usepackage{epsfig}
\usepackage{latexsym}
\usepackage{stmaryrd}
\usepackage{ae,aecompl}

\usepackage[english]{babel}
 \usepackage[colorlinks=true]{hyperref}
\usepackage{pstricks}
\usepackage{enumerate}

\newcommand{\Z}{\mathbb{Z}}
\renewcommand{\P}{\mathbb{P}}
\newcommand{\E}{\mathbb{E}}
\DeclareFixedFont{\beaupetit}{T1}{ftp}{b}{n}{2cm}

\newtheorem{theorem}{Theorem}[]
\newtheorem{definition}{Definition}[]
\newtheorem{proposition}[]{Proposition}

\newtheorem{lemma}[]{Lemma}
\newtheorem{corollary}[]{Corollary}

\theoremstyle{definition}
\newtheorem{algorithm}{Algorithm}
\newtheorem*{remark}{Remark}
\newtheorem*{notation}{Notation}

\newcommand{\m}{\mathbf{m}}
\newcommand{\Vm}{\mathrm{V}(\m)}
\newcommand{\Em}{\mathrm{E}(\m)}
\newcommand{\Fm}{\mathrm{F}(\m)}
\newcommand{\M}{\mathcal{M}}
\newcommand{\Minf}{\mathbf{M}^{(\infty)}_\q}
\newcommand{\q}{\mathsf{q}}
\newcommand{\qt}{\tilde{\mathsf{q}}}
\newcommand{\N}{\mathbb{N}}
\newcommand{\R}{\mathbb{R}}
\newcommand{\wq}{w_{\q}}
\newcommand{\Zq}{Z_{\q}}

\newcommand{\Pq}{\mathbb{P}_{\q}}
\newcommand{\Pqinf}{\mathbb{P}_{\q}^{(\infty)}}
\newcommand{\Ptinf}{\bar{\mathbb{P}}_{\q}^{(\infty)}}
\newcommand{\Pql}{\mathbb{P}^{(l)}_{\q}}

\newcommand{\fq}{f_{\q}}

\newcommand{\IBHPM}{\textup{\textsf{IBHPM}}}
\renewcommand{\rq}{r_{\q}}

\newcommand{\p}{\mathbf{p}}
\newcommand{\C}{\mathcal{C}}
\newcommand{\e}{\bar{\mathbf{e}}}

\newcommand{\A}{\mathcal{A}}
\newcommand{\Cs}{\mathsf{C}}
\newcommand{\Gs}{\mathsf{G}}
\newcommand{\I}{\mathcal{I}}
\newcommand{\Pt}{\bar{\mathbb{P}}}
\renewcommand{\c}{\mathbf{c}}

\newcommand{\cut}{\textup{\textsf{Cut}}}
\newcommand{\Hcut}{\textup{\textsf{Hcut}}}

\renewcommand{\sb}{\sigma}

\newcommand{\Br}{\partial_r\Minf}
\newcommand{\Bl}{\partial_l\Minf}

\title{Duality of random planar maps \emph{via} percolation}
\author{Nicolas Curien \footnote{Université Paris-Saclay and Institut Universitaire de France. E-mail: \texttt{nicolas.curien@gmail.com}.} \quad \text{and} \quad Lo\"ic
  Richier\footnote{\'Ecole Polytechnique. E-mail: \texttt{loic.richier@polytechnique.edu}.}} 
             
\begin{document}
            
             \maketitle
\vspace{-9mm}
             \abstract{We discuss duality properties of critical Boltzmann planar maps such that the degree of a typical face is in the domain of attraction of a stable distribution with parameter $\alpha\in(1,2]$. We consider the critical Bernoulli bond percolation model on a Boltzmann map in the dilute and generic regimes $\alpha \in (3/2,2]$, and show that the open percolation cluster of the origin is itself a Boltzmann map in the dense regime $\alpha \in (1,3/2)$, with parameter \[\alpha':= \frac{2\alpha+3}{4\alpha-2}.\] This is the counterpart in random planar maps of the duality property $\kappa \leftrightarrow 16/\kappa$ of Schramm--Loewner Evolutions and Conformal Loop Ensembles, recently established by Miller, Sheffield and Werner \cite{miller_cle_2017}. As a byproduct, we identify the scaling limit of the boundary of the percolation cluster conditioned to have a large perimeter. The cases of subcritical and supercritical percolation are also discussed. In particular, we establish the sharpness of the phase transition through the tail distribution of the size of the percolation cluster.
               }
\vspace{-2mm}
\begin{figure}[ht]
	\centering
	\includegraphics[scale=.6]{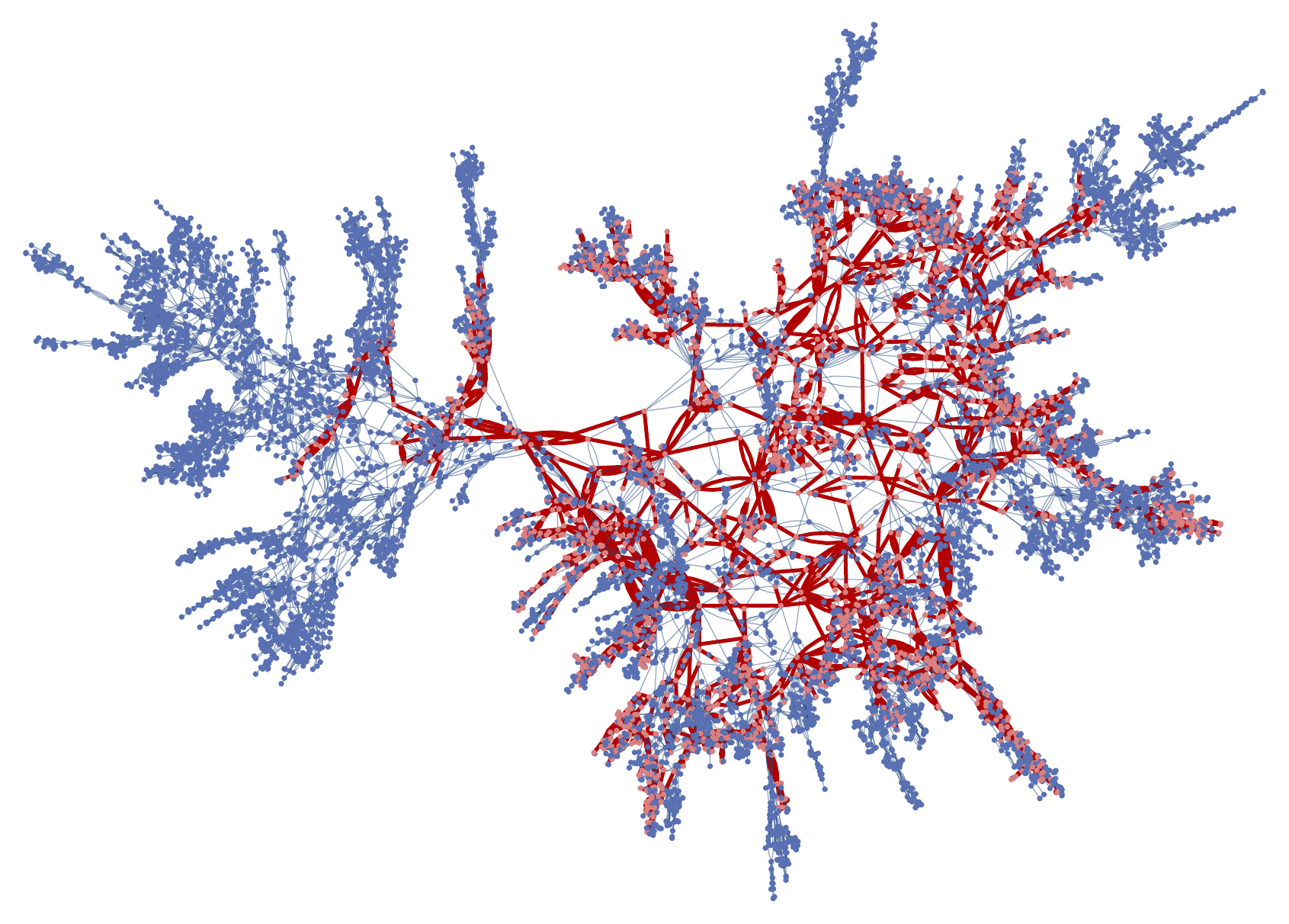}
	\vspace{-1mm}
	\caption{A Boltzmann map equipped with a critical bond percolation model. The open percolation cluster of the origin is in red.}
	\label{fig:test2}
	\end{figure}

\section{Introduction and main results}

The purpose of this work is to study duality properties of Boltzmann planar maps through the Bernoulli bond percolation model. The Boltzmann measures on planar maps are para\-metriz\-ed by a weight sequence $\q=(q_1,q_2,\ldots)$ of nonnegative real numbers assigned to the faces of the maps. Precisely, the Boltzmann weight of a \textit{bipartite} planar map $\m$ (that is, with faces of even degree) is 
\[\wq(\m):=\prod_{f \in \textup{Faces}(\m)}q_{\frac{\deg(f)}{2}}.\] The sequence $\q$ is said to be \textit{admissible} when $\wq$ gives a finite measure to the set of \textit{rooted} bipartite maps, i.e.\ with a distinguished oriented edge called the \textit{root edge}. The resulting probability measure $\Pq$ is the Boltzmann measure with weight sequence $\q$. A planar map with distribution $\Pq$ is called a $\q$-Boltzmann map, and denoted by $M_\q$.

\smallskip

The \textit{scaling limits} of Boltzmann bipartite maps conditioned to have a large number of faces have been actively studied, providing also a natural classification of the weight sequences. This classification is better understood in terms of pointed rooted Boltzmann maps, i.e., with a distinguished vertex in addition to the root edge. We generally focus on \textit{critical} weight sequences, so that the average number of edges of such a map is infinite (otherwise, the weight sequence is called \textit{subcritical}). Among critical weight sequences, special attention has been given to \textit{generic critical} sequences, for which the degree of a typical face has finite variance. Building on earlier works of Marckert \& Miermont \cite{marckert_invariance_2007} and Le Gall \cite{le_gall_uniqueness_2013}, Marzouk proved in \cite{marzouk_scaling_2016} that generic critical Boltzmann maps all have the same scaling limit, the \textit{Brownian map}. For a different scaling limit to arise, Le Gall \& Miermont initiated in \cite{le_gall_scaling_2011} the study of critical sequences $\q$ such that the degree of a typical face is in the domain of attraction of a stable law with parameter $\alpha\in(1,2)$. The weight sequence $\q$ is then called \textit{non-generic critical} with parameter $\alpha$, and the associated $\q$-Boltzmann map $M_\q$ is said to be \textit{discrete stable} with parameter $\alpha$. Under slightly stronger assumptions, they proved the subsequential convergence towards random metric spaces called the \textit{stable maps} with parameter $\alpha$ (see also~\cite{budd_geometry_2017,bertoin_martingales_2016} for a study of their dual maps). The geometry of non-generic critical Boltzmann maps exhibits large faces that remain present in the scaling limit. The behaviour of these faces is believed to differ in the dense phase $\alpha\in(1,3/2)$, where they are supposed to be self-intersecting, and in the dilute phase $\alpha\in(3/2,2)$, where it is conjectured that they are self-avoiding, see \cite{richier_limits_2017}. In this paper, we will also deal with $\q$-Boltzmann maps that are discrete stable with parameter $\alpha=2$, meaning that $\q$ is critical and that the degree of a typical face falls into the domain of attraction of a Gaussian distribution (thereby generalizing the generic critical regime).

\medskip

The framework of this paper is the Bernoulli bond percolation model on discrete stable maps with parameter $\alpha\in(3/2,2]$. Given a planar map $\m$, the bond percolation model on $\m$ is defined by declaring each edge open (or black) with probability $p\in[0,1]$ and closed (or white) otherwise, independently for all edges. When considering this model, we implicitly work conditionally on the event that the root edge of the map is open. We are mostly interested in the open connected component containing the root edge, called the (open) \textit{percolation cluster} of the origin. 

As we will detail in Section \ref{sec:BondPercolation}, there is a natural definition of critical parameter, the \textit{percolation threshold} $p_\q^c\in[0,1]$, that has been determined in \cite{curien_peeling_2016} as an explicit function of the weight sequence $\q$. Our main result deals with the distribution of the percolation cluster of the origin in Boltzmann maps.

\begin{theorem}\label{th:dualityperco}
	Let $M_\q$ be a discrete stable map with parameter $\alpha\in(3/2,2]$, equipped with a Bernoulli bond percolation model of parameter $p\in[0,1]$. Let $\C$ be the percolation cluster of the origin in $M_\q$. Then $\C$ is a Boltzmann map (conditioned to have at least one edge). Moreover, when $p=p_\q^c$, the map $\C$ is discrete stable with parameter
	\[\alpha':= \frac{2\alpha+3}{4\alpha-2}\in[7/6,3/2),\] while when $p<p_\q^c$ it is subcritical and if $p>p_\q^c$ it is discrete stable with parameter $\alpha'=\alpha$.
\end{theorem}

\begin{remark} 

Notice that the whole range of values $\alpha'$ in the dense regime is not accessible, the parameter $\alpha'=7/6$ being a lower bound. This is the parameter associated with critical percolation on discrete stable maps with parameter $\alpha=2$, see \cite{bernardi_boltzmann_2017} for the case of uniform triangulations. We interpret this result has a duality property of Boltzmann maps through critical percolation, which has also been observed in the so-called \textit{Schramm--Loewner Evolutions} (SLE) and \textit{Conformal Loop Ensembles} (CLE) by Miller, Sheffield and Werner \cite{miller_cle_2017}. By Theorem \ref{th:dualityperco}, we establish the discrete counterpart of this result in random planar maps. Note that the results of \cite{miller_cle_2017} also call on a continuum analog of critical Bernoulli percolation. This question is closely related to a stronger form of the celebrated Knizhnik--Polyakov--Zamolodchikov (KPZ) formula \cite{knizhnik_fractal_1988}. Namely, it is believed that planar maps equipped with statistical mechanics models converge towards the so-called \textit{Liouville Quantum Gravity} (LQG) model \cite{duplantier_liouville_2011} coupled with a CLE of a certain parameter $\kappa \in (8/3,8)$ (see \cite{gwynne_convergence_2017} for an example in the case of the percolation model on quadrangulations). Moreover, discrete stable maps are known to be related to planar maps equipped with a $\mathcal{O}(n)$ loop model \cite{borot_recursive_2012} through the so-called \textit{gasket decomposition}. As a consequence, there is a conjectural relation between the parameter $\alpha \in (1,2]$ of Boltzmann maps and the parameter $\kappa$ of CLEs, given by the formula 
\[\alpha=\frac{1}{2}+\frac{4}{\kappa}.\] We can thus check that the duality relation of Theorem \ref{th:dualityperco} corresponds to the duality relation for SLEs and CLEs of \cite{miller_cle_2017}, that is $\kappa'=16/\kappa$, through this identity. 

Finally, it has been established in \cite[Theorem 12]{curien_peeling_2016} that $p_{\q}^{c}=1$ for $\q$-Boltzmann maps in the subcritical and dense regimes, which explains why we consider only discrete stable maps with parameter $\alpha\in(3/2,2]$ in Theorem \ref{th:dualityperco}.
\end{remark}

The recent results of \cite{richier_limits_2017} also allow to identify the scaling limit of the boundary of the percolation cluster, in the Gromov--Hausdorff sense (see \cite{burago_course_2001} for details on this topology).

\begin{corollary}{\textup{\cite[Theorem 1.1]{richier_limits_2017}}} Let $M_\q$ be a discrete stable map with parameter $\alpha\in(3/2,2]$, equipped with a Bernoulli bond percolation model of parameter $p=p_\q^c$. Let $\C_k$ be the percolation cluster of the origin in $M_\q$, conditioned to have perimeter $2k$ (equipped with its graph distance). Then there exists a slowly varying function $\Lambda$ such that in the Gromov-Hausdorff sense,
	\[\frac{\Lambda(k)}{k^{1/\beta}}\cdot \partial \C_k \underset{k\rightarrow \infty}{\overset{(d)}{\longrightarrow}} \mathscr{L}_{\beta}, \qquad \text{where} \qquad \beta:=\alpha-\frac{1}{2} \in(1,3/2]\]and $\mathscr{L}_{\beta}$ is the random stable looptree with parameter $\beta$, see \cite{curien_random_2014}.
\end{corollary}

We also prove the following results concerning the distribution of the size of the percolation cluster (see the end of this introduction for the meaning of the notation $\approx$).

\begin{proposition}\label{prop:Sharpness}
	Let $M_\q$ be a discrete stable map with parameter $\alpha\in(3/2,2]$, equipped with a Bernoulli bond percolation model of parameter $p\in[0,1]$. Let $\C$ be the percolation cluster of the origin in $M_\q$, and $\vert \C \vert $ be its total number of vertices. Then the following estimates hold.
	
\begin{itemize}
	\item Critical case. If $p=p_\q^c$, we have
\[\Pq(\vert \C \vert  = 2n)\underset{n \rightarrow \infty}{\approx} n^{-\frac{8\alpha+4}{2\alpha+3}}.\]

\item Supercritical case. If $p_\q^c<p\leq 1$, we have
\[\Pq(\vert \C \vert  = 2n)\underset{n \rightarrow \infty}{\approx} n^{-\frac{2\alpha+1}{\alpha}}.\]

\item Subcritical case. If $0\leq p<p_\q^c$, there exists $C_1,C_2>0$ such that for every $n \in \Z_{\geq 0}$,
\[\Pq(\vert \C \vert  = 2n)\leq C_1\exp(-C_2 n).\]

\end{itemize}
\end{proposition}

Note that in the previous statement, we implicitly restricted ourselves to the values of $n$ for which $\Pq(\vert \C \vert  = 2n)\neq 0$.

\begin{remark} In the wording of statistical mechanics, this result is known as the \textit{sharpness of the phase transition} of the Bernoulli bond percolation model, since the tail distribution of the volume of the percolation cluster is polynomial in the (super)critical regime, and suddenly becomes exponential in the subcritical regime. In the case of critical and supercritical percolation, Proposition \ref{prop:Sharpness} is a direct consequence of Theorem \ref{th:dualityperco} and the Bouttier--Di Francesco--Guitter bijection \cite{bouttier_planar_2004} (see Proposition \ref{prop:degree}). On the contrary, in the subcritical case, Proposition \ref{prop:Sharpness} does not follow from Theorem \ref{th:dualityperco}. There, we also need to prove the exponential decay for the perimeter of the percolation cluster (see Lemma \ref{lem:ExponentialTailPerimeter}), from which Proposition \ref{prop:Sharpness} and Theorem \ref{th:dualityperco} both stem.\end{remark}

\paragraph{Context.} This paper takes place in the framework of percolation on random planar maps, which has been the subject of extensive work. The first model to be considered was site percolation on the so-called \textit{Uniform Infinite Planar Triangulation} ($\textsf{UIPT}$), for which Angel determined the percolation threshold in \cite{angel_growth_2003}. Later on, he also dealt with the half-plane analog of the $\textsf{UIPT}$ \cite{angel_scaling_2004}. This work was then extended to bond and face percolation models and to quadrangulations of the half-plane \cite{angel_percolations_2015} and of the plane \cite{menard_percolation_2014}. The case of site percolation on quadrangulation, more delicate, has also been studied \cite{bjornberg_site_2015,richier_universal_2015}. Furthermore, many properties of the percolation models have been investigated, like critical exponents \cite{angel_percolations_2015,gorny_geometry_2017}, crossing probabilities \cite{richier_universal_2015} or the \textit{Incipient Infinite Cluster} \cite{richier_incipient_2017}. The common thread of these papers is the so-called \textit{peeling process}, that first appeared in the pioneer work of Watabiki \cite{watabiki_construction_1995} and was made rigorous by Angel in \cite{angel_growth_2003}. The peeling process proved to be very effective to study random planar maps, in particular the percolation models on them. When introduced by Angel, the rough idea of the peeling process was to reveal the map face by face. More recently, Budd introduced in \cite{budd_peeling_2016} a variant called the \textit{lazy peeling process}, that reveals instead the map \textit{edge by edge} and allows to treat all Boltzmann maps in a unified way \cite{curien_peeling_2016}. Note that percolation on random planar maps have been investigated by means of other techniques, like combinatorial decompositions \cite{curien_percolation_2014} and analytic combinatorics \cite{bernardi_boltzmann_2017}. 

\medskip Our approach in this paper builds upon and unifies those of \cite{angel_percolations_2015} and of \cite{bernardi_boltzmann_2017}. More precisely, as in   \cite{bernardi_boltzmann_2017} we use the Boltzmann point of view --in particular the notion of criticality for Boltzmann maps-- but replace the analytic combinatorics part by new probabilistic estimates using the peeling process. However, although our result encompasses a large class of maps, namely bond percolation on bipartite Boltzmann maps, it does not recover those of \cite{bernardi_boltzmann_2017} which were obtained in the case of (bond and site) percolation on triangulations (where the maps are non-bipartite). We believe that our methods extend to the non-bipartite case, at the cost of technical difficulties.

As mentioned above, the key tool is to use a new peeling process which is tailored to bond percolation exploration. This process is close in spirit to those of \cite{angel_growth_2003} or \cite{angel_percolations_2015} but slightly more delicate here. The strategy of our proof is based on the study of \textit{cut-edges} of percolation clusters, that disconnect the cluster when removed, see Figure~\ref{fig:Cut}. By computing the probability that the root edge is a cut-edge of the percolation cluster in two different ways, we are able to determine the partition function of the cluster, which is enough to characterize discrete stable maps as we shall prove. We then transpose the problem on \textit{Infinite Boltzmann Half-Planar Maps}, in which the exploration process is described in terms of a random walk with increments in the domain of attraction of a stable law, for which the probabilistic estimates are routine. \medskip 

The paper is organized as follows. In Section \ref{sec:DiscreteStableMaps}, we start with fundamental definitions around Boltzmann maps, set up some tools that we need, and then give a clear account of the possible classifications of Boltzmann maps (Propositions \ref{prop:TypeSequences} and \ref{prop:TypeSequences2})  that we prove to be all equivalent to each other. Section \ref{sec:BondPercolation} is devoted to background on the bond percolation model on Boltzmann maps, where we define from scratch the peeling processes that we use throughout the paper. We then turn in Section \ref{sec:CutEdges} to the technical core of the paper, which is to estimate the probability that the root edge is a cut-edge of the percolation cluster in the half-planar case. This allows to establish Theorem \ref{th:dualityperco} in the cases of critical and supercritical percolation. Finally, we deal with subcritical percolation in Section \ref{sec:Subcritical}, where we prove Theorem \ref{th:dualityperco} and Proposition \ref{prop:Sharpness} by using a peeling process defined on $\q$-Boltzmann maps instead of their half-planar version. 

\bigskip

\noindent\textbf{Notation.} Throughout the paper, given two sequences of real numbers $(x_n : n \in \N)$ and $(y_n : n \in \N)$, we write
\[x_n \underset{n \rightarrow \infty}{\approx} y_n\] if there exists an (eventually positive) slowly varying function $L$ such that $x_n=L(n)y_n$ for every $n \geq 0$. Recall that a function $L : \R_+ \rightarrow \R_+$ is slowly varying (at infinity) if for every $\lambda>0$ we have $L(\lambda x)/L(x) \rightarrow 1$ as $x \rightarrow \infty$.

We may also use the notation $x_n \propto y_n$ if there exists a constant $C>0$ (that does not depend on $n$) such that $x_n=C y_n$ for every $n\geq 0$. 

\bigskip

\noindent \textbf{Acknowledgements.} This work was supported by the grants ANR-15-CE40-0013 (ANR Liouville), ANR-14-CE25-0014 (ANR GRAAL) and the ERC GeoBrown. We warmly thank Grégory Miermont for interesting discussions.

\section{Boltzmann maps}\label{sec:DiscreteStableMaps}

This section is devoted to Boltzmann distributions and their properties. The goals are to set up tools for the remainder of the paper, as well as to connect the several possible definitions of the type of a Boltzmann map that are scattered over the literature. There are three commonly used classifications of weight sequences and Boltzmann maps: one uses the Bouttier--Di Francesco--Guitter bijection between planar maps and well-labeled trees \cite{bouttier_planar_2004}, another deals with the partition function of maps with a boundary, and the last one is based on the peeling process \cite{budd_peeling_2016}. We show that these definitions are equivalent in Proposition \ref{prop:TypeSequences} and \ref{prop:TypeSequences2}.

\smallskip

\subsection{Maps} A \textit{planar map} is a proper embedding of a finite connected graph in the two-dimensional sphere $\mathbb{S}^2$, viewed up to orientation-preserving homeomorphisms. We always consider rooted maps, i.e., with a distinguished oriented edge $e_*$ called the root edge. The faces of a map are the connected components of the complement of the embedding of the edges, and the degree $\deg(f)$ of the face $f$ is the number of oriented edges incident to this face. For technical reasons, we restrict ourselves to bipartite maps, in which all the faces have even degree. All the maps we consider are planar, rooted and bipartite so that we may simply call such an object a map. A generic map is usually denoted by $\m$, and we use the notation $\Vm$, $\Em$ and $\Fm$ for the sets of vertices, edges and faces of $\m$. The set of maps is denoted by $\M$. We may also consider \textit{pointed} maps, which have a marked vertex and whose set is denoted by $\M^\bullet$.

We will also deal with maps \textit{with a boundary}, meaning that the \textit{root face} $f_*$ that lies on the right of the root edge is interpreted as an \textit{external face}, whose incident edges and vertices form the \textit{boundary} of the map (while the other vertices, edges and faces are called \textit{internal}). The boundary $\partial\m$ of a map $\m$ is called \textit{simple} if it is a cycle without self-intersections. The degree $\#\partial\m$ of the external face is the \textit{perimeter} of $\m$. The set of maps with perimeter $2k$ is denoted by $\M_{k}$. By convention, the map $\dagger$ made of a single vertex is the only element of~$\M_0$.
\subsection{Boltzmann distributions}\label{sec:BoltzmannDistributions} For every \textit{weight sequence} $\q=(q_k : k\in \N)$ of nonnegative real numbers, we define the \textit{Boltzmann weight} of a map $\m\in\M$ by
\begin{equation}\label{eqn:BoltzmannWeightBoundary}
	\wq(\m):=\prod_{f \in \Fm}q_{\frac{\deg(f)}{2}},
\end{equation} with the convention $\wq(\dagger)=1$. We say that $\q$ is admissible if the \textit{partition function} $\wq(\M)$ is finite. Surprisingly, this is equivalent to say that $\wq\left(\M^\bullet\right)$ is finite (see \cite[Proposition 4.1]{bernardi_boltzmann_2017} for a proof). We denote by $\Pq$ the probability measure on $\M$ associated to $\wq$, and call~$\q$-Boltzmann map a map with this distribution. We will use the following function introduced in~\cite{marckert_invariance_2007}:
\begin{equation}\label{eqn:fq}
	\fq(x):= \sum_{k=1}^{\infty}\binom{2k-1}{k-1}q_k x^{k-1}, \quad x\geq 0.
\end{equation} By \cite[Proposition 1]{marckert_invariance_2007}, $\q$ is admissible if and only if the equation $\fq(x)=1-1/x$ has a solution in $(0,\infty)$, and the smallest such solution is then denoted by $\Zq$. This number has a fairly nice interpretation since $
	\Zq =(\wq\left(\M^\bullet\right)+1)/2$. The equation $\fq(\Zq)= 1 - \frac{1}{\Zq}$ enables us to define the following probability distribution.
	\begin{definition}\label{def:JSLaw}
If $\q$ is an admissible weight sequence, we put
\[ \mu_\q(0)=1-\fq(\Zq) \quad \text{and} \quad  \mu_\q(k)=(\Zq)^{k-1}\binom{2k-1}{k-1}q_{k}, \quad k\in \N.\]
\end{definition}

We let $m_{\mu_\q}$ be the mean of the probability measure $\mu_\q$, which yields a first classification of weight sequences essentially due to \cite{marckert_invariance_2007,le_gall_scaling_2011,borot_recursive_2012}. This distribution pops-up if one studies (pointed) Boltzmann maps using bijections with so-called \textit{well-labeled trees}. Indeed, through the Bouttier--Di Francesco--Guitter bijection \cite{bouttier_planar_2004} together with the Janson--Stef\'ansson bijection \cite[Section 3]{janson_scaling_2015}, one obtains a Galton--Watson tree whose offspring distribution is precisely $\mu_{\q}$, see \cite[Proposition 7]{marckert_invariance_2007}, \cite[Appendix A]{janson_scaling_2015} and \cite[Lemma 2.2]{richier_limits_2017} for details. In particular, we always have $m_{\mu_\q}\leq 1$.

\begin{definition}\label{def:Classification}
	Let $\q$ be an admissible weight sequence. The $\q$-Boltzmann map $M_\q$ is called critical if $\mu_\q$ has mean $m_{\mu_\q}=1$, and subcritical otherwise.
	
Moreover, we say that $M_\q$ is discrete stable of parameter $\alpha\in(1,2]$ if $\mu_\q$ is critical and in the domain of attraction of a stable distribution with parameter $\alpha$.
\end{definition}

\begin{remark} Recall that $\mu_\q$ is in the domain of attraction of a stable distribution with parameter $\alpha\in (1,2)$ if \[\mu_\q\left([k,\infty)\right) \underset{k \rightarrow \infty}{\approx} k^{-\alpha},\] while $\mu_\q$ is in the domain of attraction of a Gaussian distribution (stable with parameter $2$) if the truncated variance
	\[V_\q(k):=\sum_{j=0}^k{j^2\mu_\q(j)}\] is slowly varying at infinity. By Karamata's Theorem \cite[Theorem 8.1.6]{bingham_regular_1989}, in both cases, this is equivalent to the existence of a slowly varying function $\ell_\q$ such that 
	\begin{equation}\label{eqn:DefBassinAttraction}
		\varphi_{ \mu_\q}(t)=1-t+\ell_\q\left(1/t\right)t^{\alpha} + o(t^{\alpha} \ell_\q\left(1/t\right)) \quad \text{as } t\rightarrow 0^+ \quad (\alpha \in (1,2]),
	\end{equation} where $\varphi_{ \mu_\q}$ is the Laplace transform of the probability measure $\mu_\q$.
\end{remark}

We emphasize that Definition \ref{def:Classification} is slightly more general than those of \cite{marckert_invariance_2007,le_gall_scaling_2011,borot_recursive_2012,le_gall_uniqueness_2013,budd_geometry_2017,bertoin_martingales_2016,curien_peeling_2016}, because we allow slowly varying corrections. See \cite[Remark 2.5]{richier_limits_2017} for details on the interpretation of these definitions in terms of $\q$-Boltzmann maps. Before moving to another characterization of criticality, let us state a result, which combined with Theorem \ref{th:dualityperco} will imply the critical and supercritical cases of Proposition \ref{prop:Sharpness}.

\begin{proposition} \label{prop:degree} Let $M_\q$ be a discrete stable map of parameter $\alpha \in (1,2]$. Then, its total number of vertices $ \# \mathrm{V}(M_{\q})$ satisfies
\[ \Pq\left( \# \mathrm{V}(M_{\q}) = n\right) \underset{n \rightarrow \infty}{\approx} n^{- \frac{2\alpha +1}{\alpha}}.\] 
\end{proposition}
 
In this result, we implicitly restrict ourselves to values of $n$ for which $\Pq\left( \# \mathrm{V}(M_{\q}) = n\right)\neq 0$.

\begin{proof} We invoke the Bouttier--Di Francesco--Guitter bijection \cite{bouttier_planar_2004}, together with the Janson--Stef\'ansson bijection \cite[Section 3]{janson_scaling_2015} that we both mentioned. These allow to represent the number of vertices of a \textit{pointed} $\q$-Boltzmann map $M^\bullet_{\q}$ (chosen in $\M^\bullet$ with probability $\P_{\q}^\bullet$ proportional to $w_{\q}$) as the total number of \textit{leaves} of a Galton--Watson tree whose offspring distribution $\mu_{\q}$ (given in Definition \ref{def:JSLaw}) is critical and falls within the domain of attraction of a stable law with parameter $\alpha$. Combined with the results of \cite{kortchemski_invariance_2012}, we obtain for those $n \in \N$ for which the probability is non-zero that
\[\P_{\q}^\bullet\left( \#\mathrm{V}(M^\bullet_{\q}) = n \right) \underset{n \rightarrow \infty}{\approx} n^{-\frac{1}{\alpha}-1}.\] But we also have 
\[\P_{\q}\left( \#\mathrm{V}(M_{\q}) = n \right) \propto \frac{1}{n} \P_{\q}^\bullet\left( \#\mathrm{V}(M^\bullet_{\q})= n \right),\] and the proof follows. \end{proof}

\subsection{Boltzmann maps with a boundary}\label{sec:BoltzmannMapsBoundary} Let us now introduce the partition functions for maps with a fixed perimeter,
\begin{equation}\label{eqn:PartitionFunctionFk}
	W_\q^{(k)}:=\frac{1}{q_k}\sum_{\m\in\M_{k}}\wq(\m), \quad k\in \Z_{\geq 0},
\end{equation} where the factor $1/q_k$ stands for the fact that the root face receives no weight. These quantities are all finite when $\q$ is admissible. The distribution of $\q$-Boltzmann maps with perimeter $2k$ is then defined by 

\begin{equation}\label{eqn:BoltzmannBoundary}
	\Pq^{(k)}(\m):=\frac{\mathbf{1}_{\{\m \in \M_k \}}\wq(\m)}{q_kW_\q^{(k)}}, \quad \m\in\M, \ k\in \Z_{\geq 0}.
\end{equation} We now recall the asymptotics of the partition function $W_\q^{(k)}$. Following \cite{borot_recursive_2012} (see also \cite[Section 5.1]{curien_peeling_2016} or \cite[Section 2.2]{richier_limits_2017}), the asymptotics of the partition function $W_\q^{(k)}$ are given by 

\begin{equation*}
\begin{array}{cccc}
	W_\q^{(k)} &\underset{k\rightarrow \infty}{\sim}&\displaystyle\frac{C_\q \rq^{-k}}{k^{\frac{3}{2}}} & (M_\q \ \textup{subcritical}) \\
	\\
W_\q^{(k)} &\underset{k\rightarrow \infty}{\approx}& \displaystyle\frac{\rq^{-k}}{k^{\alpha+\frac{1}{2}}} & (M_\q \ \textup{discrete stable with parameter }\alpha \in(1,2])
\end{array}
\end{equation*} where $\rq:=(4\Zq)^{-1}$ and $C_\q$ is a positive constant. Note that for $\alpha=2$, our more general definition requires a new proof that we provide in Proposition \ref{prop:TypeSequences2}. The exponent $a:=\alpha+1/2$ dictates the polynomial behaviour of the partition function. Thus, we borrow the notation of~\cite{curien_peeling_2016} for weight sequences.

\begin{notation}
An admissible weight sequence $\q$ is of type $a=3/2$ if the $\q$-Boltzmann map $M_\q$ is subcritical, and of type $a\in(3/2,5/2]$ if $M_\q$ is discrete stable with parameter $\alpha=a-1/2$.
\end{notation}

It turns out that the asymptotics of the partition function can be used as an equivalent definition for the type of a weight sequence, as we will see in Propositions \ref{prop:TypeSequences} and \ref{prop:TypeSequences2}.

\subsection{Peeling process of $\q$-Boltzmann maps}\label{sec:PeelingqBoltzmann}

We start by reviewing the lazy peeling process of $\q$-Boltzmann maps, first introduced in \cite[Section 3.1]{budd_peeling_2016}. We refer to \cite[Chapter 3]{curien_peeling_2016} for a detailed presentation. 

Let us consider a finite map $\m\in\M$. We call \textit{peeling exploration} of $\m$ an increasing sequence $\e_0 \subset \e_1 \subset \cdots \subset \m$ of sub-maps of $\m$ that contain the root edge. We also require that the map $\e_i$ has a marked simple face, called the \textit{hole}, such that when filled-in with the proper (and unique) map with a general boundary, we recover $\m$. A peeling exploration of the map $\m$ is driven by an \textit{algorithm} $\A$ that associates to every map $\e_i$ an edge on the boundary of its hole, or a cemetery state $\ddag$. The cemetery state is interpreted as the end of the peeling exploration (in particular, if the hole has perimeter zero, $\A$ must take the value~$\ddag$). We denote by $\theta$ be the lifetime of the peeling exploration. 

 The peeling exploration of $\m$ driven by $\A$ is the sequence of sub-maps $\e_0 \subset \cdots \subset \e_{\theta} \subset \m$ defined as follows. First, the map $\e_0$ is made of a simple face whose degree is that of the root face $f_*$ of $\m$. Then given $\e_i$, the map $\e_{i+1}$ is obtained by revealing the status of the face $f_i$ incident to the left of $\A(\e_i)$ in $\m$, which we call \textit{peeling} the edge $\A(\e_i)$. Two situations may occur, that are illustrated in Figure \ref{fig:PeelingFinite}:

\begin{enumerate}
	\item The face $f_i$ does not belong to $\e_i$, and has degree $2k$ $(k\in \N)$. In this case, $\e_{i+1}$ is obtained from $\e_i$ by adding this new face, without performing any identification of its edges (that is, we simply add a polygon of degree $2k$ incident to $\A(\e_i)$). We denote this event by $\Cs_k$.
	\item The face $f_i$ belongs to $\e_i$. In this case, $\e_{i+1}$ is obtained from $\e_i$ by identifying the two half-edges of the hole that correspond to $\A(\e_i)$. This creates two holes, whose perimeters are denoted by $2j$ and $2k$ $(j, k \in \Z_{\geq 0})$, from left to right. In this case, we fill-in one of the holes (determined by the algorithm $\mathcal{A}$) using the proper map with a boundary, so that $ \e_{i+1}$ has one hole. We denote this event by $\Gs_{j,k}$.
\end{enumerate}
This peeling process is called a \textit{filled-in} exploration in the terminology of \cite{budd_peeling_2016,curien_peeling_2016}.

\begin{figure}[ht]
	\centering
	\includegraphics[scale=1.1]{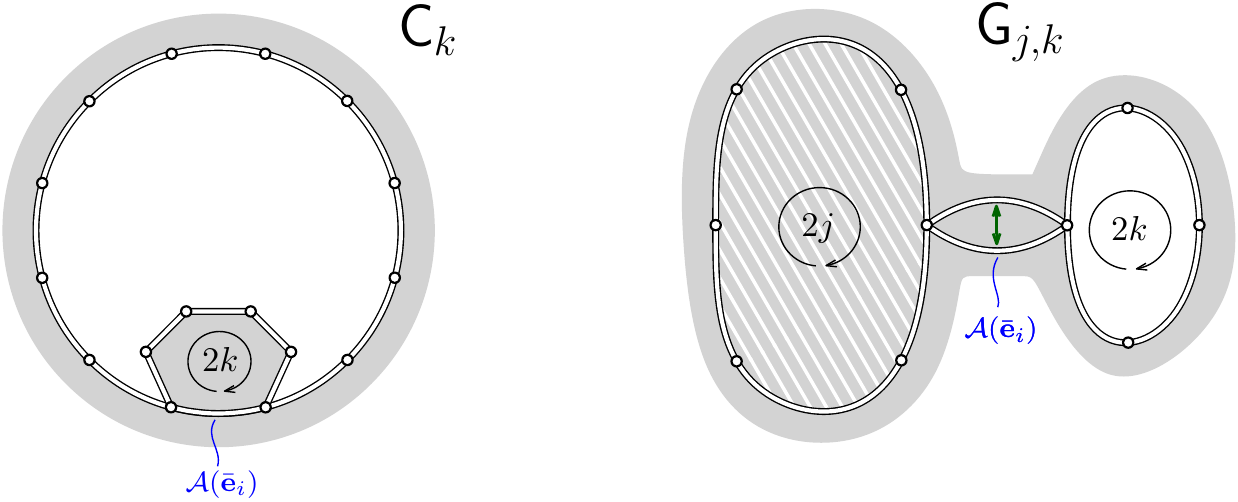}
	\caption{The peeling events $\Cs_k$ and $\Gs_{j,k}$. The explored part of the map is in gray and the unexplored part in white. The filled-in hole is the hatched area.}
	\label{fig:PeelingFinite}
	\end{figure}

\smallskip

We now let $\q$ be an admissible weight sequence, and give the distribution of the steps of the peeling process $(\e_i : 0 \leq i \leq \theta)$ of a $\q$-Boltzmann map (that is, under $\Pq$). This distribution does not depend on the peeling algorithm $\A$, that can even be random as long as it does not use information on the unrevealed parts of the map. For every $0\leq i \leq \theta$, we let $P_i$ be the half-perimeter of the hole of the map $\e_i$. Then conditionally on the map $\e_i$ (and on the event $\{\theta>i\}$) we have
\begin{equation}\label{eqn:LawPeelingFinite}
\Pq\left(\Cs_k \mid \e_i \right)=\p^{(P_i)}(k) \quad (k\geq 1) \qquad \text{and} \qquad \Pq\left(\Gs_{j,k} \mid \e_i \right)= \p^{(P_i)}(j,k) \quad(j,k\geq 0), 
\end{equation} where
\[\p^{(l)}(k):=q_k \displaystyle\frac{W_\q^{(l+k-1)}}{W_\q^{(l)}} \ \ (l\geq 1, \ k \geq 1) \quad \text{and} \quad \p^{(l)}(j,k):=\displaystyle\frac{W_\q^{(j)}W_\q^{(k)}}{W_\q^{(l)}}\mathbf{1}_{j+k+1=l} \ \ (l\geq 1, \ j,k \geq 0) \] Finally, at each time, the maps that fill in the holes are $\q$-Boltzmann maps with the proper perimeter, independent of the past exploration (this is sometimes referred to as the \textit{spatial Markov property}). We refer to \cite[Proposition 7]{curien_peeling_2016} for detailed proofs.

Note that the probabilities in \eqref{eqn:LawPeelingFinite} depend on $\e_i$ only through the perimeter of the hole. Thus, by \cite[Lemma 6]{curien_peeling_2016}, when the half-perimeter $P_{i}$ tends to infinity, these probabilities converge towards limiting transition probabilities $\p^{(\infty)}(k)= \nu_{\q}(k-1)$ and $\p^{(\infty)}(\infty,k) = \p^{(\infty)}(k,\infty) = \frac{1}{2} \nu_{\q}(-k-1)$, where 
\begin{equation}\label{eqn:DefNu}
\nu_\q(k)=
\left\lbrace
\begin{array}{cc}
q_{k+1}\rq^{-k}  & \mbox{ if } k\geq 0\\
2W_\q^{(-k-1)}\rq^{-k} & \mbox{ if } k\leq -1\\
\end{array}\right..
\end{equation} We refer to \cite[Lemma 9]{curien_peeling_2016} for the proof that this is indeed a probability measure on $ \mathbb{Z}$.  The quantities $ \p^{(\infty)}$ will be interpreted as the transition probabilities for the peeling process on the $\q$-Boltzmann map of the half-plane (see Section \ref{sec:IBHPM}). 

\subsection{Classification of weight sequences} In the last sections we have defined the offspring distribution $\mu_{\q}$ appearing when dealing with bijections with labeled trees, the partition functions $W_{\q}$ enumerating maps with a boundary and the probability measure $\nu_{\q}$ connected to the peeling process. We now give equivalent definitions of the type of a weight sequence $\q$ using those notions, starting with the criticality property. \medskip 

The following result is a direct consequence of \cite[Proposition 4]{budd_peeling_2016} and \cite[Proposition 4.3]{bernardi_boltzmann_2017}. Recall that a random walk $(X_i : i\geq 0)$ \textit{drifts} to $\infty$ (resp.\ to $-\infty$) if $P(X_i\geq X_0, \ \forall i\geq 0)>0$ (resp.\ if $P(X_i\leq X_0, \ \forall i\geq 0)>0$). Moreover, if a (non-constant) random walk drifts neither to $\infty$ nor to $-\infty$, it is said to \textit{oscillate}.

\begin{proposition}\label{prop:TypeSequences} Let $\q$ be an admissible weight sequence and $M_\q$ be a $\q$-Boltzmann map. Then the following statements are equivalent.
\begin{enumerate}
	\item The probability measure $\mu_\q$ has mean $m_{\mu_\q}<1$ ($M_\q$ is subcritical, $\q$ is of type $3/2$).
	\item There exists constants $C>0$ and $r>0$ such that $W_\q^{(k)}\sim C r^{-k}k^{-3/2}$ as $k\rightarrow \infty$.
	\item The random walk $X_\q$ with steps distributed as $\nu_\q$ drifts to $-\infty$.
	\item The quantity $\E_\q[\#\mathrm{V}(M_\q)^2]$ is finite.
\end{enumerate} Moreover, if none of these conditions is satisfied, $\mu_\q$ is critical and $X_\q$ oscillates.
\end{proposition}

The next result deals with the characterizations of the type of a critical weight sequence.

\begin{proposition}\label{prop:TypeSequences2}Let $\q$ be an admissible weight sequence such that the $\q$-Boltzmann map $M_\q$ is critical, and fix $a \in(3/2,5/2]$. Then the following statements are equivalent.

\begin{enumerate}
	\item The probability measure $\mu_\q$ is in the domain of attraction of a stable law with parameter $a-1/2$ ($M_\q$ is discrete stable with parameter $a-1/2$, $\q$ is of type $a$).
	\item The partition function $W_\q$ satisfies $W_\q^{(k)} \approx \rq^{-k}k^{-a}$ as $k\rightarrow \infty$.
	\item The probability measure $\nu_\q$ satisfies $\nu_\q([k,\infty))\approx k^{1-a}$ as $k\rightarrow \infty$.
	\item The probability measure $\nu_\q$ satisfies $\nu_\q((-\infty,k])\approx k^{1-a}$ as $k\rightarrow -\infty$.
\end{enumerate}
\end{proposition}

\begin{proof} The implication $\textit{1}$$\Rightarrow$$\textit{2}$ has already been established in \cite[Equation (15)]{richier_limits_2017}, except for the case $a=5/2$ where we use here a more general definition. The weight sequence $\q$ being admissible, we have the following expression for the partition function \cite{borot_recursive_2012,curien_peeling_2016}
\begin{equation}\label{eqn:IntegralFk}
	W_\q^{(k)}=\binom{2k}{k} \int_0^1 \left(uZ_{\q(u)}\right)^k \mathrm{d}u, \quad k\in \Z_{\geq 0},
\end{equation} where $\q(u):=(u^{k-1}\q_k : k\in \N)$. Following \cite[Section 2.2]{richier_limits_2017}, we can rewrite the integral as \begin{equation}\label{eqn:IntegralFk2}\int_0^1 \left(uZ_{\q(u)}\right)^k \mathrm{d}u=(\Zq)^k \int_0^\infty e^{-kt}U(\mathrm{d}t),\end{equation} where \[U(t):=-\Zq e^{-t} + \Zq \varphi_{\mu_{\q}}(t) , \quad t\geq 0,\] and $\varphi_{\mu_\q}$ is the Laplace transform of the probability measure $\mu_\q$. By applying \cite[Theorem 1.7.1']{bingham_regular_1989} to the function $U$ in \eqref{eqn:IntegralFk2} together with the remark \eqref{eqn:DefBassinAttraction} on the domain of attraction of stable laws, this ensures that $\textit{1}$$\Leftrightarrow$$\textit{2}$.

By the definition of $\mu_\q$ in Lemma \ref{def:JSLaw} and that of $\nu_\q$ in \eqref{eqn:DefNu}, we obtain that
\[\nu_\q(k)=\frac{4^k}{{2k-1\choose k-1}}\mu_\q(k+1), \quad k\in\Z_{\geq 0}.\] A summation by parts then establishes that $\textit{1}$$\Leftrightarrow$$\textit{3}$. Finally, by \eqref{eqn:DefNu}, we have $\nu_\q(k)=2W_\q^{(-k-1)}\rq^{-k}$ if $k\leq -1$ which shows that $\textit{2}$$\Leftrightarrow$$\textit{4}$ and concludes the proof.\end{proof}

\section{Bond percolation and explorations}\label{sec:BondPercolation}

In this section, we define the bond percolation model and the associated peeling exploration. The bond percolation model on a map $\m$ is defined by declaring each edge open (or black) with probability $p\in[0,1]$ and closed (or white) otherwise, independently for all edges. For the sake of clarity, we usually hide the parameter $p$ in the notation. When considering this percolation model, we implicitly work conditionally on the event that the root edge of the map is black (but the states of the other edges remain unknown). We will be interested in the black connected component of the source of the root edge, called the black percolation cluster of the origin. In what follows, we may simply speak of percolation cluster since there is no risk of confusion. Naturally, the root edge of the percolation cluster is that of the underlying map. 

We first prove the easy fact that the percolation cluster of a Boltzmann map is still Boltzmann distributed. We then turn to the description of our peeling exploration both in a finite and infinite setup.

\bigskip

\subsection{Percolation clusters are Boltzmann maps} Let $\q$ be an admissible weight sequence, and consider the percolation model with parameter $p\in[0,1]$ on a $\q$-Boltzmann map. We first establish that the percolation cluster is itself a Boltzmann map. We rely on a decomposition inspired by \cite{bernardi_boltzmann_2017}. This decomposition is based on \textit{islands}, that are maps with a simple boundary in which the boundary edges are black, and all the internal edges incident to the boundary are white. For every $k\in\N$, we let $\I_k$ be the set of islands with perimeter $2k$ and introduce the partition function

\begin{equation}\label{eqn:PartFctIslands}
	I_k(p):=\frac{1}{q_k}\sum_{\m\in\I_k}{p^{\#\mathrm{E}_\bullet(\m)}(1-p)^{\#\mathrm{E}_\circ(\m)}\wq(\m)},
\end{equation} where $\mathrm{E}_\bullet(\m)$ (resp.\ $\mathrm{E}_\circ(\m)$) stands for the set of black (resp.\ white) internal edges of $\m$. 

The fundamental observation is that a percolated map $\m$ can be uniquely decomposed into its percolation cluster $\c$ and a collection of islands $(\m_f : f\in\mathrm{F}(\c))$ associated to the faces of the cluster, and such that the perimeter of $\m_f$ equals $\deg(f)$ for every $f\in\mathrm{F}(\c)$. These islands are obtained by cutting the map $\m$ along the edges of the cluster, see \cite[Section 2.3]{bernardi_boltzmann_2017} for details and Figure \ref{fig:DecompositionB} for an illustration. For this decomposition to be unique, a rooting convention of the islands has to be chosen (any deterministic procedure is suitable). We obtain the following result.

\begin{proposition}\label{prop:LawCluster}
	The percolation cluster $\C$ of $M_\q$ is a $\qt$-Boltzmann map (conditioned to have at least one edge), where $\qt$ is defined by
	\begin{equation}\label{eqn:DefQt}
		\qt_k=p^kI_k(p), \quad k\in\N.
	\end{equation} In particular, the weight sequence $\qt$ is admissible.
\end{proposition}

\begin{proof}The proof follows from the above decomposition, which yields \[\Pq(\C=\c)\propto p^{\#\mathrm{E}(\c)}\prod_{f\in \mathrm{F}(\c)}I_{\frac{\deg(f)}{2}}(p),\quad \c\in\M\backslash\{\dagger\}.\] By the identity $\#\mathrm{E}(\c)=\sum_{f\in \mathrm{F}(\c)}{\tfrac{1}{2}\deg(f)}$, this gives
	\[\Pq(\C=\c)\propto w_{\qt}(\c),\quad \c\in\M\backslash\{\dagger\},\] and thus the expected result. At the same time, this shows that $\qt$ is admissible.\end{proof}

\begin{figure}[ht]
	\centering
	\includegraphics[scale=1.4]{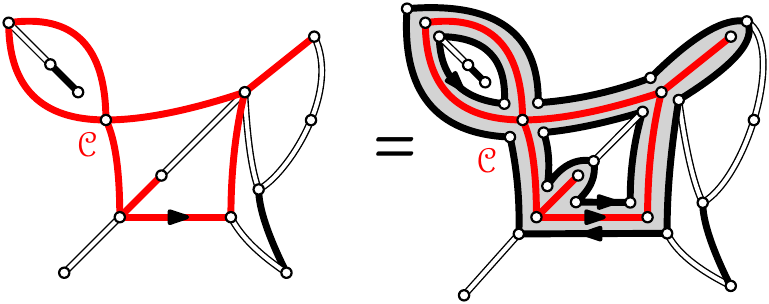}
	\caption{The decomposition of a map with perimeter $12$ into its percolation cluster $\C$ with perimeter $8$ (in red) and three islands.}
	\label{fig:DecompositionB}
	\end{figure}

	We now prove a result that will be useful in Section \ref{sec:RelateTh1Cut} to relate percolation on finite $\q$-Boltzmann maps to percolation on $\q$-Boltzmann maps of the half-plane.

When the map $\m$ and its percolation cluster $\c$ both have fixed perimeter, say $2l$ and $2m$ respectively, the above decomposition can still be defined. As before, every internal face $f$ of $\c$ is filled in with an island of perimeter $\deg(f)$. However, the root face of $\c$ plays a special role: it is filled in with an island of perimeter $2m$ such that the face incident to the left of the root edge has degree $2l$ (because of the prescribed perimeter of $\m$, see Figure \ref{fig:DecompositionB}). For every $l,m\in\N$, we let $\I_m^{(l)}$ be the set of such islands and denote by $I_m^{(l)}(p)$ the associated partition function, defined as in \eqref{eqn:PartFctIslands}. The following result extends Proposition \ref{prop:LawCluster} to Boltzmann maps with a boundary.

\begin{proposition}\label{prop:LawClusterB} Let $l,m\in\N$, and consider the percolation model with parameter $p\in[0,1]$ on a $\q$-Boltzmann map $M_\q^{(l)}$ with perimeter $2l$ (i.e., with law $\Pql$). Then under $\Pql(\cdot \mid \#\partial\C=2m)$, the percolation cluster $\C$ of $M_\q^{(l)}$ has law $\P_{\qt}^{(m)}$, where $\qt$ is defined by \eqref{eqn:DefQt}.
\end{proposition}

\begin{proof}By the aforementioned decomposition of maps with fixed perimeter, for every $\c\in\M_m$,
\[\Pql(\C=\c)\propto I_m^{(l)}(p) p^{\#\mathrm{E}(\c)}\prod_{\substack{f\in \mathrm{F}(\c) \\ f \neq f_*(\c)}}I_{\frac{\deg(f)}{2}}(p),\] so that there exists a constant $\kappa$ (that depends on $p$, $\q$, $m$ and $l$ but not on $\c$) such that $\Pql(\C=\c)=\kappa \cdot w_{\qt}(\c)$. By summing over all maps $\c\in\M_m$, we have
\[\Pql(\#\partial\C=2m)=\sum_{\c\in\M_m}\Pql(\C=\c)=\kappa \cdot \qt_m W_{\qt}^{(m)},\] so that $\Pql(\C=\c\mid \#\partial\C=2m)=\P_{\qt}^{(m)}(\c)$ for every $\c\in\M_m$, as expected.\end{proof}

\subsection{Peeling exploration of bond percolation}
We now introduce the peeling algorithm that we will use in order to study the bond percolation model on $\q$-Boltzmann maps.
\subsubsection{Percolation exploration on finite deterministic maps}

We start by defining this algorithm in a \textit{deterministic} context. Let us consider a finite map $\m\in\M$, and assume that every edge $e$ of $\m$ carries an nonnegative integer number $n_e$ of ``marks" (represented by red crosses on our figures). We turn these marks into a coloring of the edges by declaring that the edge $e$ is black if and only if $n_e>0$. The exploration of $\m$ will be a sequence of decorated sub-maps $\e_0 \subset \cdots \subset \e_\theta \subset \m$, i.e.~sub-maps carrying marks on their edges, including the boundary edges of their hole (we however keep the same notation $\e$ for a sub-map). Here, $\e_i$ is a sub-map of $\m$ means that we can recover $\m$ by gluing inside the only hole of $\e_i$ the proper map with a general boundary, this map carrying itself marks on the edges, and we simply add-up the marks in the gluing operation (see Figure \ref{fig:GlueCross}).

\begin{figure}[!h]
 \begin{center}
 \includegraphics[width=0.75\linewidth]{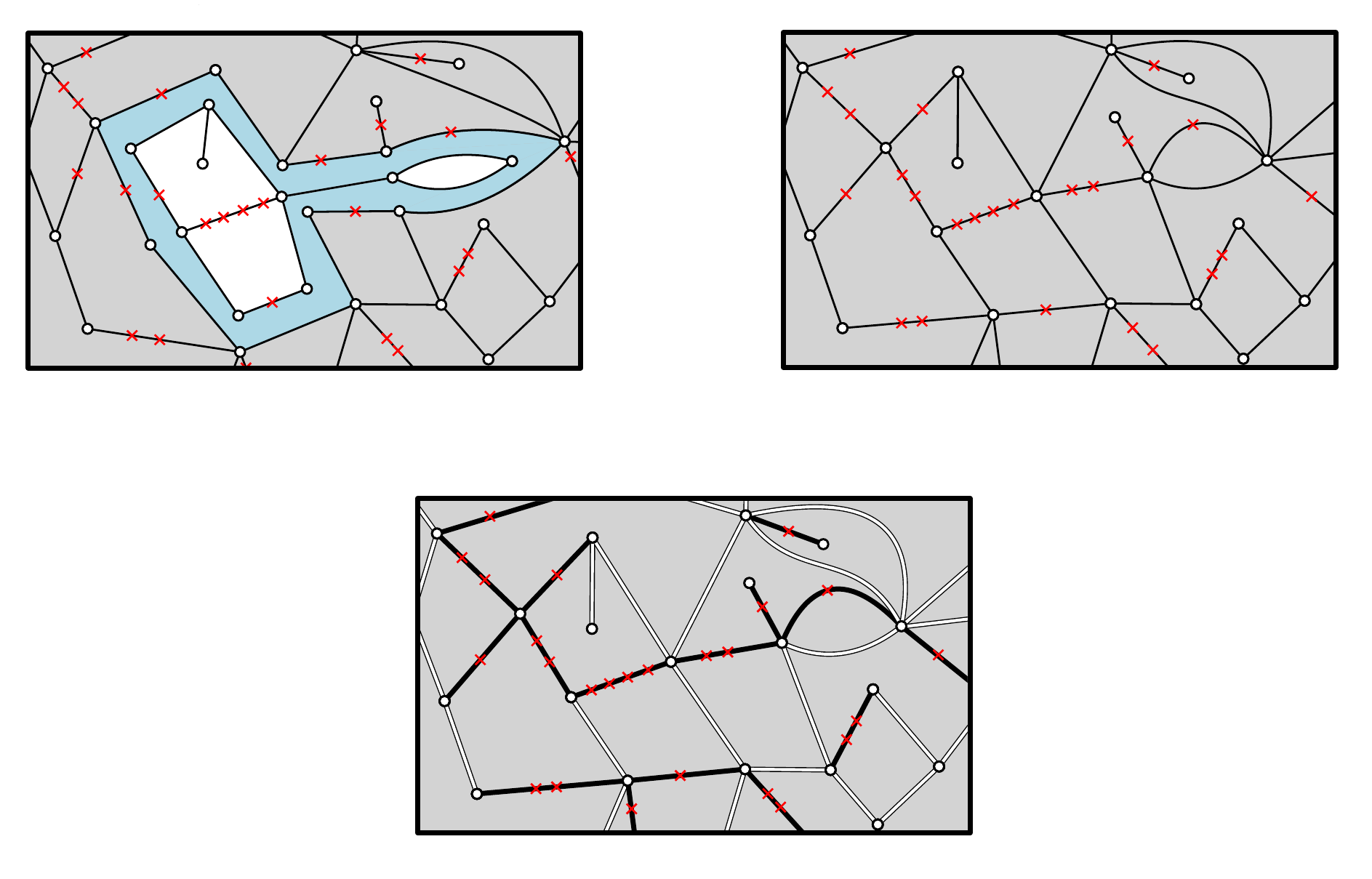}
 \caption{Illustration of the gluing operation in the presence of marks. Notice in particular that the marks carried by the half-edges on both sides of the ``isthmus-edge'' of the inside map add up in the gluing process. The bottom picture shows the interpretation of marks in terms of percolation.}
 \label{fig:GlueCross}
 \end{center}
 \end{figure}

In the next part, we call \textit{boundary condition} of $\e_i$ the marks on the boundary of its hole. We will say that $\e_i \subset \m$ has a ``black-free'' boundary condition if the boundary of its unique hole is made of two finite connected segments, one made of edges carrying a single mark (black edges)  and the other one of edges carrying no mark (free edges). When $\e_i \subset \m$ has a ``black-free'' boundary condition, then the edge to peel $ \mathcal{A}( \e_i)$ is the edge immediately on the left of the black segment (when turning counterclockwise inside the hole). Furthermore, if the boundary condition is totally black, then we set $\A(\e_i) = \ddag$ and the algorithm ends.

The basic principle of the peeling process $(\e_i : 0 \leq i \leq \theta)$ of $\m$ is to reveal if $n_e=0$ or not, without revealing the exact value $n_e$ itself. Precisely, the algorithm works as follows: we start by defining $\e_0$ as the map made of a simple face whose degree is that of the root face of $\m$, and carrying a single mark on the edge that corresponds to the root edge (recall that we assumed that the root edge carries at least one mark). The sub-map $\e_{0} \subset \m$ thus has a black-free boundary condition, and the edge $\A(\e_0)$ is the edge on the left of the root edge.

\begin{algorithm}\label{alg:PeelingpercoFinite} Let $0\leq i < \theta$ and assume that the map $\e_i$ has ``black-free" boundary condition. Recall that $\A(\e_i)$ is the edge on the left of the black segment on the boundary of the hole. We consider the edge $\epsilon$ of the map that fills in the hole of $\e_i$ and that is opposite to $\A(\e_i)$.

\begin{enumerate}

\item If $\epsilon$ carries at least one mark, then remove one mark from $\epsilon$ and add this mark on $\A(\e_i)$ to form $\e_{i+1}$ (we just move one mark, the total number of marks being unchanged).
\item If $\epsilon$ carries no mark, then we trigger a standard peeling step and reveal the face inside the hole that is incident to $\A(\e_i)$. If we discover a new face (event of type $ \mathsf{C}_{k}$) the new edges on the boundary of the hole carry no mark. If two half-edges corresponding to $\A(\e_i)$ are identified (event of type $ \mathsf{G}_{j,k}$), their marks add up and two holes are created. Then, we fill-in the hole that has a \textit{totally monochromatic boundary}.	
	\end{enumerate}
\end{algorithm}

Note that the ``black-free" boundary condition is preserved, so that Algorithm \ref{alg:PeelingpercoFinite} is well defined.  In case 2, for an event of type $\Cs_k$, both sides of the peeled edge have been discovered and no mark has been encountered: this edge is white, see Figure~\ref{fig:PeelingPercoFinite}.  For an event of type $\Gs_{j,k}$, the peeled edge may be identified to a black edge (hence become black) or to a free edge (hence become white). The two cases are illustrated in Figure~\ref{fig:PeelingPercoFinite}.

\begin{figure}[ht]
	\centering
	\includegraphics[scale=.8]{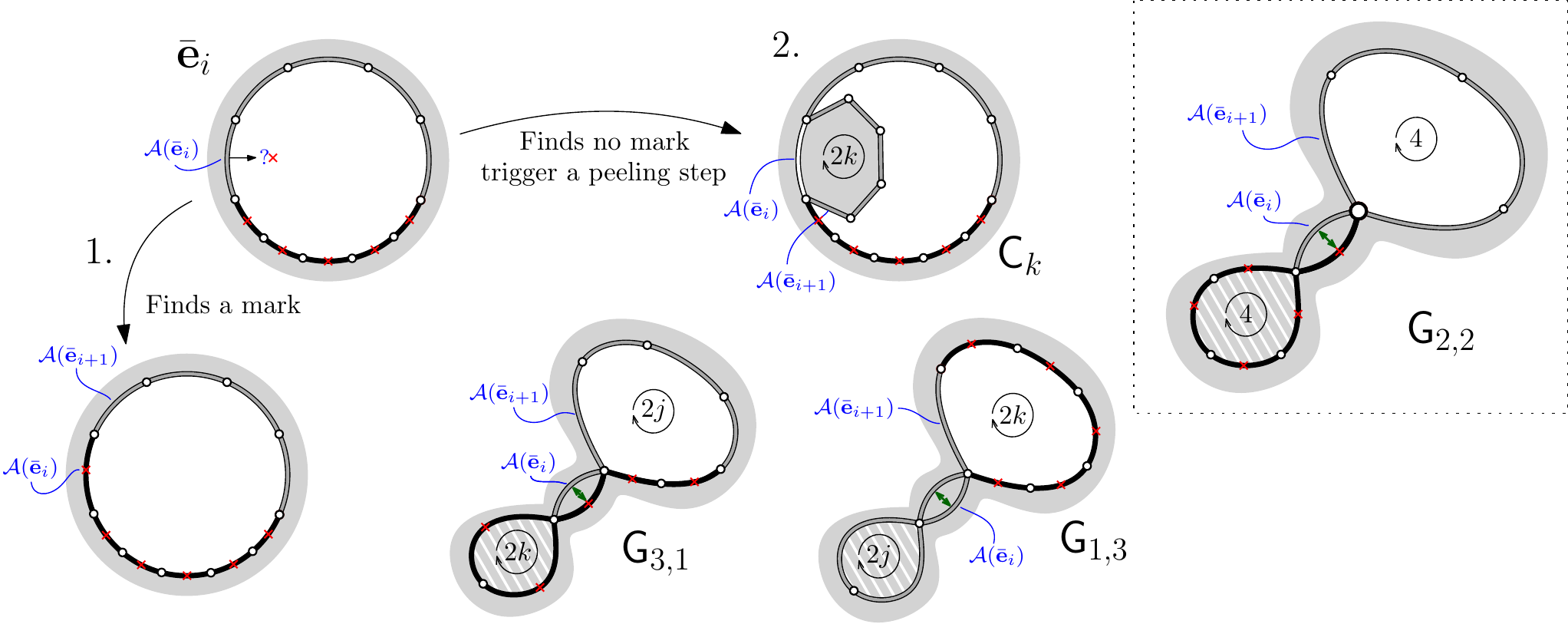}
	\caption{The peeling Algorithm \ref{alg:PeelingpercoFinite}. Free edges are represented in darkgray, the explored regions in lightgray  and the filled-in regions are hatched. In the top right drawing, the large vertex serves as a black segment of length zero on the boundary of the hole.}
	\label{fig:PeelingPercoFinite}
	\end{figure}

We now make important remarks. Some care is needed when in case 2, an event of type $ \mathsf{G}_{j,k}$ identifies the peeled edge with the right-most edge of the black boundary, or to the free edge on the right of the black boundary. In the first case, the convention is that the endpoint of the peeled edge belonging to the hole with free boundary serves as a black boundary of length zero. So we fill in the other hole, and continue the exploration in the free hole as shown in the framed case of Figure~\ref{fig:PeelingPercoFinite}. In the second case, the convention is to fill-in the free hole. This is one of the two ways the algorithm can stop, both being illustrated in Figure~\ref{fig:PeelingPercoFiniteStop}.

	\begin{figure}[h!]
	\centering
	\includegraphics[scale=.82]{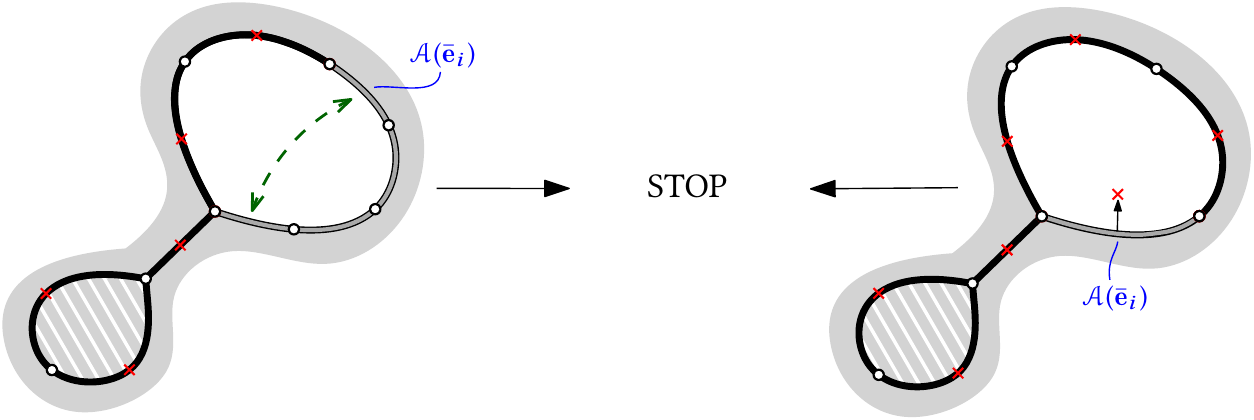}
	\caption{The two ways Algorithm \ref{alg:PeelingpercoFinite} may stop. (Left) The $\theta$-th peeling step identifies $ \mathcal{A}( \e_{\theta-1})$ with the edge adjacent immediately on the right of the black segment. Then, we fill-in the ``free'' hole created. (Right) At time $\theta-1$ the boundary of $\e_{\theta-1}$ contains a single free edge which is then given a mark at step $\theta$ and becomes black. In both cases, the boundary of $\e_{\theta}$ is completely black, and the algorithm stops. 
	\label{fig:PeelingPercoFiniteStop}}
	\end{figure}

When running, the exploration driven by Algorithm \ref{alg:PeelingpercoFinite} goes clockwise around the boundary $\partial\mathcal{C}$ of the percolation cluster, starting from the root edge (see Figure \ref{fig:examplecomplete}). In particular, at time $\theta$, the boundary $\partial\C$ of the percolation cluster has been completely revealed by the exploration process, as shown in Figure \ref{fig:PeelingPercoFiniteStop}. Since at every step of the peeling exploration, at most one half-edge of $\partial\C$ is discovered, we have the crude bound
\begin{equation}\label{eqn:BoundPerimeterCLuster}
	\#\partial\C\leq \theta+1,
\end{equation} that will be useful later on (where the extra factor one accounts for the half root-edge). Before moving to the stochastic properties of this process when run on percolated $\q$-Boltzmann maps, we let the reader get accustomed to this exploration by performing the full exploration of the cluster given in Figure~\ref{fig:examplecomplete} below.

\begin{figure}[!h]
 \begin{center}
 \includegraphics[width=0.9\linewidth]{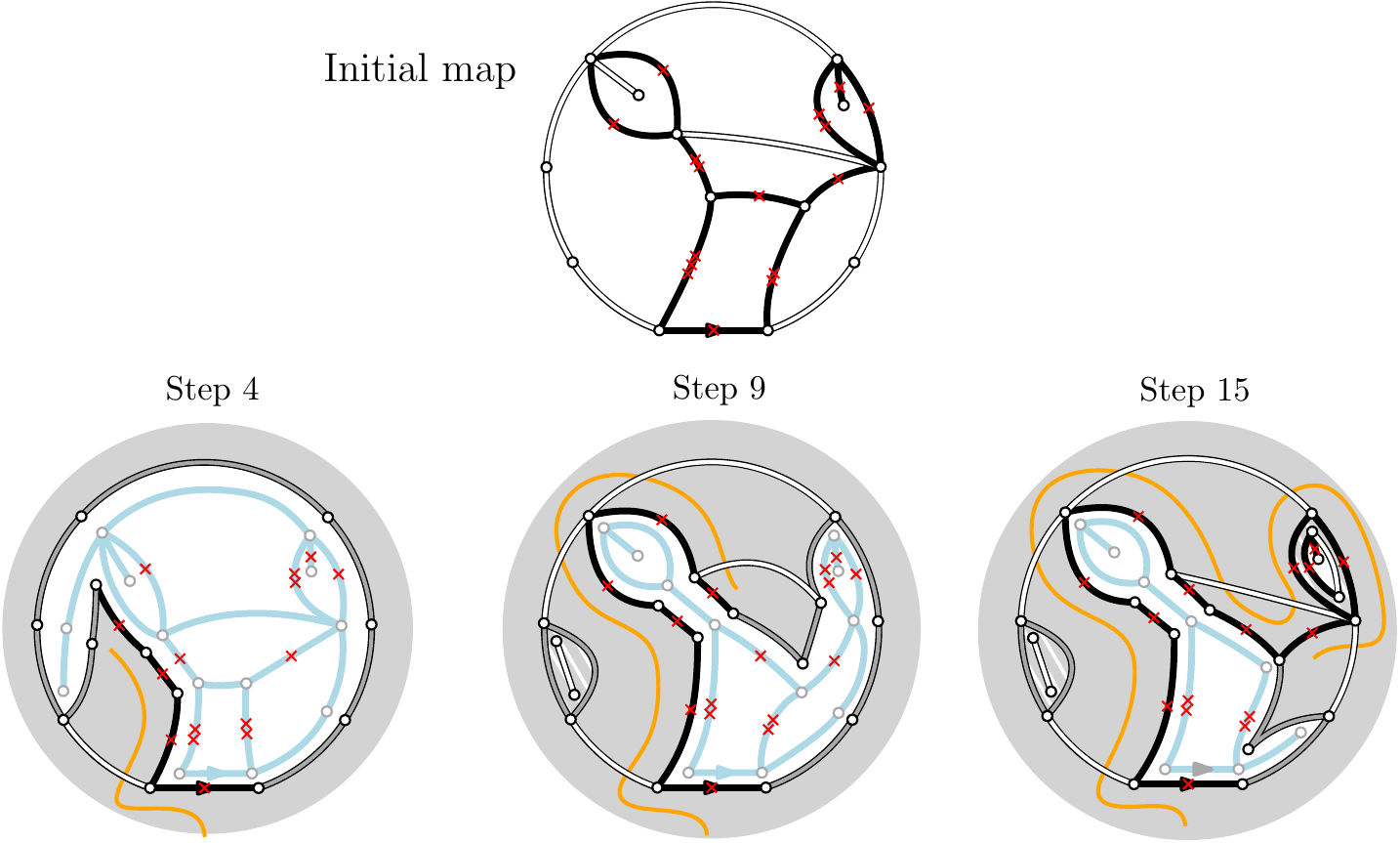}
 \caption{A step by step exploration of a percolated map. On top, the decorated map to explore as well as its interpretation as a bond percolation model. Below, the exploration frozen at step 4, step 9 and step 15. The explored region is in lightgray with filled-in parts hashed. The region that remains to be explored is in white and we drew the map filling-it in light blue. In orange is the path followed by the exploration: it turns clockwise around $ \mathcal{C}$. Notice the difference between the two ``isthmus" edges visited at steps 9 and 15: in the first case, this edge carries more than one mark, so the peeling exploration will discover its two half-edges as black and continue without identifying them. In the second case, it carries a single cross so that when the second half-edge is visited, a peeling step is performed and the two half-edges are glued together.}
 \label{fig:examplecomplete}
 \end{center}
 \end{figure}
 
 \subsubsection{Percolation exploration on $ \q$-Boltzmann maps} \label{sec:peelingBoltfinite}

 We now let $\q$ be an admissible weight sequence, consider a $\q$-Boltzmann map $M_\q$. We assume that every edge $e$ of $M_\q$ carries an independent random number $N_e$ of marks, distributed as a geometric variable with parameter $p\in[0,1]$, that is, $P(N_e=k)=(1-p)p^k$ for every $k\in\Z_{\geq 0}$. As above, these marks are interpreted as a bond percolation model by declaring that the edge $e$ is black if and only if $N_e>0$. If we further give an additional mark to the root edge (in order to force it to be black), this precisely corresponds to the bond percolation model we are working with in this paper. We keep the notation $\Pq$ for the resulting probability measure, hiding the dependence in the marks.

We now consider the peeling exploration $(\e_i : 0 \leq i \leq \theta)$ of $M_\q$ driven by Algorithm \ref{alg:PeelingpercoFinite}. We claim that during such an exploration, the following statement holds.

\begin{lemma}[Decorated spatial Markov property]\label{lem:DecoratedSMP} Let $0 \leq i < \theta$, and recall that $P_i$ is the half-perimeter of the unique hole of $\e_{i} \subset \m$. Then, conditionally on $\e_i$, the map filling-in its hole has law $ \mathbb{P}^{(P_i)}_{\q}$, and its edges are equipped with i.i.d.\ numbers of marks with geometric distribution of parameter $p$.
\end{lemma}
\proof The claim is proved by induction. First, it clearly holds in $\e_0$. Then, at step $i<\theta$, the distribution of the map filling-in the hole of $\e_{i}$ follows from the spatial Markov property of the peeling process, see \cite[Proposition 7]{curien_peeling_2016}. Moreover, in case $1$ of Algorithm \ref{alg:PeelingpercoFinite}, the induction hypothesis is checked by the memorylessness property of geometric variables: if $X$ is a geometric variable, conditionally on $X\geq 1$, we have that $X-1$ has the same law as $X$.\endproof

The above spatial Markov property of the exploration process given by Algorithm \ref{alg:PeelingpercoFinite} can be used to describe its transition probabilities only in terms of the lengths of the black and free boundaries. We first introduce some notation. For every $0 \leq i \leq \theta$, we let $B_i$ (resp.\ $F_i$) be the number of black (resp.\ free) edges on the boundary of the hole of $\e_i$ (at step $i$ of the peeling algorithm). By definition of Algorithm \ref{alg:PeelingpercoFinite}, the lifetime $\theta$ of the peeling process then reads
\begin{equation}\label{eqn:EqnTauStarFinite}
	\theta:=\inf\left\lbrace i> 0 : \A(\e_i) = \ddag\right\rbrace = \inf\{i> 0 : F_i = 0 \}.
\end{equation} Recall the subtlety that when a peeling step swallows the whole black segment except one vertex (i.e., when $B_i=0$), we agree that the black segment on the boundary is reduced to this vertex in the definition of Algorithm \ref{alg:PeelingpercoFinite}. Note that $F_0=\deg(f_*)-1$ is given by the (random) degree of the root face $f_*$ of $M_\q$, while $B_0=1$. Finally, recall that $P_i=\tfrac{1}{2}(B_i+F_i)$ stands for the half-perimeter of the hole of the map $\e_i$. 

We now examine the possible cases of Algorithm~\ref{alg:PeelingpercoFinite}. In case 1, we see that $B_{i+1}=B_i+1$ and $F_{i+1}=F_i-1$ (a free edge is turned into a black edge). In case 2, if $\Cs_k$ is realized, then $F_{i+1}=F_i+2k-1$, while if $\Gs_{j,k}$ is realized, then $F_{i+1}=F_i-2j-2$ if free edges are swallowed, or $B_{i+1}=B_i-2k-1$ and $F_{i+1}=F_i-1$ if black edges are swallowed.

Now, using the law of the peeling steps of a $\q$-Boltzmann map recalled in \eqref{eqn:LawPeelingFinite} together with the spatial Markov property of Lemma \ref{lem:DecoratedSMP}, we obtain the following result.

\begin{lemma}\label{lem:LawBFinite} Let $(\e_i : 0 \leq i \leq \theta)$ be the peeling exploration of $M_\q$ driven by Algorithm \ref{alg:PeelingpercoFinite}. Then the process $((B_i,F_i) : 0 \leq i \leq \theta)$ is a Markov chain (killed at the first time $F_i$ takes value zero) whose transition probabilities are given conditionally on $(B_{i},F_{i})$ and $\{F_i>0\}$ by
\begin{equation*}
	(B_{i+1},F_{i+1})=(B_{i},F_{i})+\left\lbrace \begin{array}{lcll}
		(1,-1) & \mbox{\footnotesize proba.} & p &\\
		(0,2k-1) & \mbox{\footnotesize proba.} & (1-p) \p^{(P_i)}(k) & (k\geq 1) \\
		(-2k-1,-1) & \mbox{\footnotesize proba.} & (1-p) \p^{(P_i)}(P_i-k-1,k) & (0 \leq k \leq \tfrac{1}{2}(B_i-1))\\
		(0,-2j-2) & \mbox{\footnotesize proba.} & (1-p) \p^{(P_i)}(j,P_i-j-1) & (0 \leq j \leq \tfrac{1}{2}F_i-1)\\
	\end{array} \right..
\end{equation*}
\end{lemma} 

\subsubsection{Bond percolation on Infinite Boltzmann Half-Planar Maps}\label{sec:IBHPM}

The proof of our main theorem is based on a careful analysis of cut-edges and of the tail distribution of the perimeter $\#\partial \mathcal{C}$ of the cluster $\C$ under $ \mathbb{P}_{\q}$. Although such an analysis should in principle be performed on finite $\q$-Boltzmann maps, it is easier to carry out in the half-planar setting, where the peeling steps lose their dependency in the perimeter and become i.i.d.. In this section, we describe the adaptation of our Algorithm \ref{alg:PeelingpercoFinite} in the half-planar case (the main difference being that the algorithm is stopped as soon as the percolation cluster $ \mathcal{C}$ is separated from infinity). \medskip

\paragraph{The $\q$-$\IBHPM$ and its peeling process.} Let us briefly review the construction of the Infinite Boltzmann Half-Planar Map and of its peeling process. We refer to  \cite[Section 4.1]{curien_peeling_2016} for details. Fix $\q$ an admissible weight sequence and recall that  $\Pq^{(k)}$ is the law of a $\q$-Boltzmann map with perimeter $2k$. Then we have the following convergence
\begin{equation}\label{eqn:qIBHPM}
	\Pq^{(k)} \underset{k\rightarrow \infty}{\Longrightarrow} \Pqinf,
\end{equation} in distribution for the  \textit{local topology}. A map with distribution $\Pqinf$ is usually denoted by $\Minf$, and called the Infinite Boltzmann Half-Planar Map with weight sequence $\q$ (for short, $\q$-$\IBHPM$). It is a.s.~a map of the half-plane, meaning that it is one-ended with a unique infinite face (the root face), that we think of as an infinite boundary. The $\q$-$\IBHPM$ is also amenable to a peeling process, see \cite[Section 4.1]{curien_peeling_2016}. The description is very similar to that of the peeling process of a $\q$-Boltzmann map and we only highlight the differences. A peeling process of a half-planar map $\m_\infty$ is an increasing sequence $(\e_i : i\geq 0)$ of sub-maps of $\m_\infty$ that contain the root edge and such that for every $i\in \Z_{\geq 0}$, the map $\e_i$ has a unique hole with infinite perimeter. Here, the map $\e_0$ is the embedding of the graph of $\Z$ in the plane (defining two infinite faces, one of which is the root face) and given $\e_i$, the map $\e_{i+1}$ is obtained by revealing the status of the face $f_i$ incident to the left of $\A(\e_i)$ in $\m_\infty$. Two situations may occur, corresponding to peeling events that we denote by $\Cs_k$ and $\Gs_{k,\infty}$ (or $\Gs_{\infty,k}$) and that are illustrated in Figure \ref{fig:Peeling}. The main difference with the finite setting is that when two holes are created, we fill-in the \textit{finite hole}.

\begin{figure}[ht]
	\centering
	\includegraphics[scale=1.1]{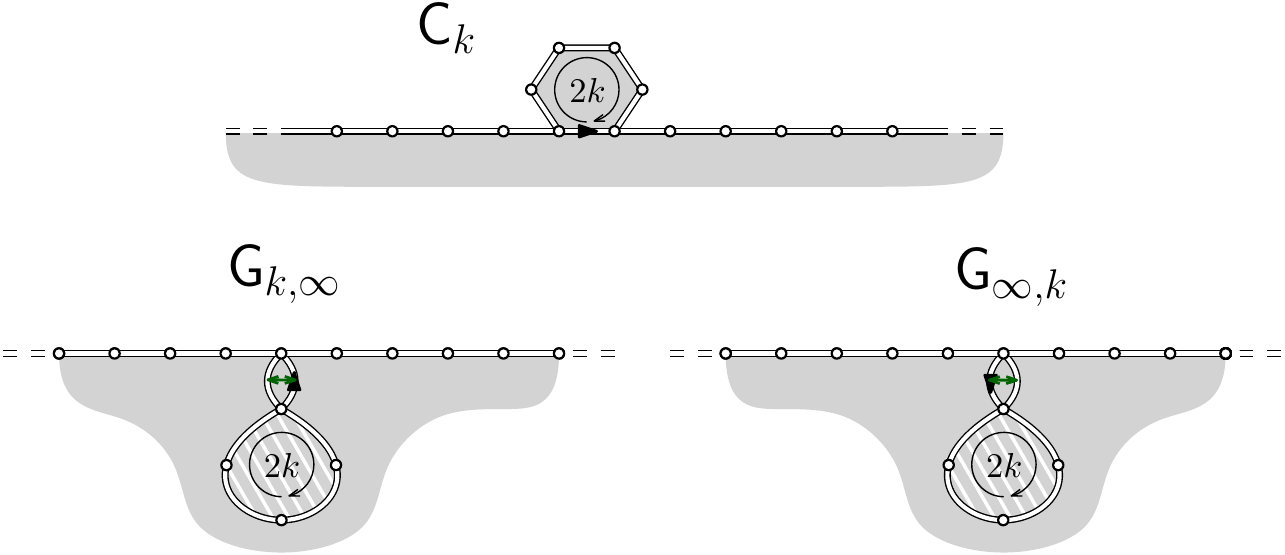}
	\caption{The peeling events $\Cs_k$, $\Gs_{k,\infty}$ and $\Gs_{k,\infty}$. On the event $\Cs_{k}$, the face $f_i$ does not belong to $\e_i$, and has degree $2k$ $(k\in \N)$. In this case, $\e_{i+1}$ is obtained from $\e_i$ by adding this new face, without performing any identification of its edges. Otherwise, $f_i$ belongs to $\e_i$ and  $\e_{i+1}$ is obtained from $\e_i$ by identifying the two half-edges of the hole that correspond to $\A(\e_i)$. This creates two holes, one of which is infinite. We denote this event by $\Gs_{k,\infty}$ (if we identify $\A(\e_i)$ to an edge on its left) or $\Gs_{\infty,k}$ (in the other case), where $2k$ $(k\in\Z_{\geq 0})$ stands for the perimeter of the finite hole. }
	\label{fig:Peeling}
	\end{figure}

\medskip

We now consider the peeling process on the $\q$-$\IBHPM$ by choosing $\m_{\infty} = \Minf$. Then, the conditional distribution of the events $\Cs_k$ and $\Gs_{k,\infty}$ does not depend on $\e_{i}$ (as long as the peeling algorithm $\mathcal{A}$ does not use information outside of $\e_{i}$) and are given by 
\begin{equation}\label{eqn:LawPeeling}
\begin{array}{ccc}
\Pqinf\left(\Cs_k \mid \e_i \right)=\nu_\q(k-1) &(k\geq 1).\\
\\
\Pqinf\left(\Gs_{k,\infty} \mid \e_i \right)=\Pqinf\left(\Gs_{\infty,k} \mid \e_i \right)= \nu_\q(-k-1)/2 &(k\geq 0),\\
\end{array}
\end{equation} where $\nu_\q$ has been defined in \eqref{eqn:DefNu}.

\paragraph{Percolation exploration in the $\q$-$\IBHPM$.} We resume with a framework similar to that of Section \ref{sec:peelingBoltfinite}, and suppose that the edges of $\Minf$ are decorated by i.i.d.~numbers of marks with geometric law of parameter $p$,  and that the root edge is given an extra mark to force it to be black. We again keep the notation $ \mathbb{P}_{ \mathbf{q}}^{(\infty)}$ for the law of the resulting map. We say that a sub-map $\e_i$ of $\Minf$ has ``free-black-free" boundary condition if the boundary of its infinite hole contains a finite segment of black edges carrying a single mark (and the other edges are free). We then explore the (left) bond percolation interface starting from the root edge through the following analogue of Algorithm \ref{alg:PeelingpercoFinite}.

\begin{algorithm}\label{alg:Peelingperco} Let $i \in \Z_{\geq 0}$ and assume that the map $\e_{i}$ has ``free-black-free" boundary condition.  Let $ \mathcal{A}( \e_{i})$  be the edge on the left of the black segment on the boundary of the hole. We consider the edge $\epsilon$ of the map that fills in the hole of $\e_i$ and that is opposite to $\A(\e_i)$.
\begin{enumerate}
\item If $\epsilon$ carries at least one mark, then remove one mark from $\epsilon$ and add this mark on $\A(\e_i)$ to form $\e_{i+1}$.
\item If $\epsilon$ carries no mark, then we trigger a standard peeling step and reveal the face inside the hole that is incident to $\A(\e_i)$. If we discover a new face (event of type $ \mathsf{C}_{k}$) the new edges on the boundary of the hole carry no mark. If two half-edges corresponding to $\A(\e_i)$ are identified (event of type $ \mathsf{G}_{k,\infty}$ or $ \mathsf{G}_{\infty,k}$), their marks add up and two holes are created. Then, we fill-in the \emph{finite} hole.\end{enumerate}
\end{algorithm}

As announced, the main difference with Algorithm \ref{alg:PeelingpercoFinite} is that we always fill-in the finite hole, and so we may (this happens actually in most cases) do not complete a full turn around the boundary $ \partial \mathcal{C}$ of the cluster before stopping the exploration. More precisely, when an event of type $ \mathsf{G}_{\infty,k}$ happens with $2k$ strictly larger than the length of the black boundary, then the latter is completely swallowed in the finite hole and the exploration stops. The lifetime of this peeling exploration is denoted by $\tau$.

\begin{figure}[h!]
	\centering
	\includegraphics[width=0.9\linewidth]{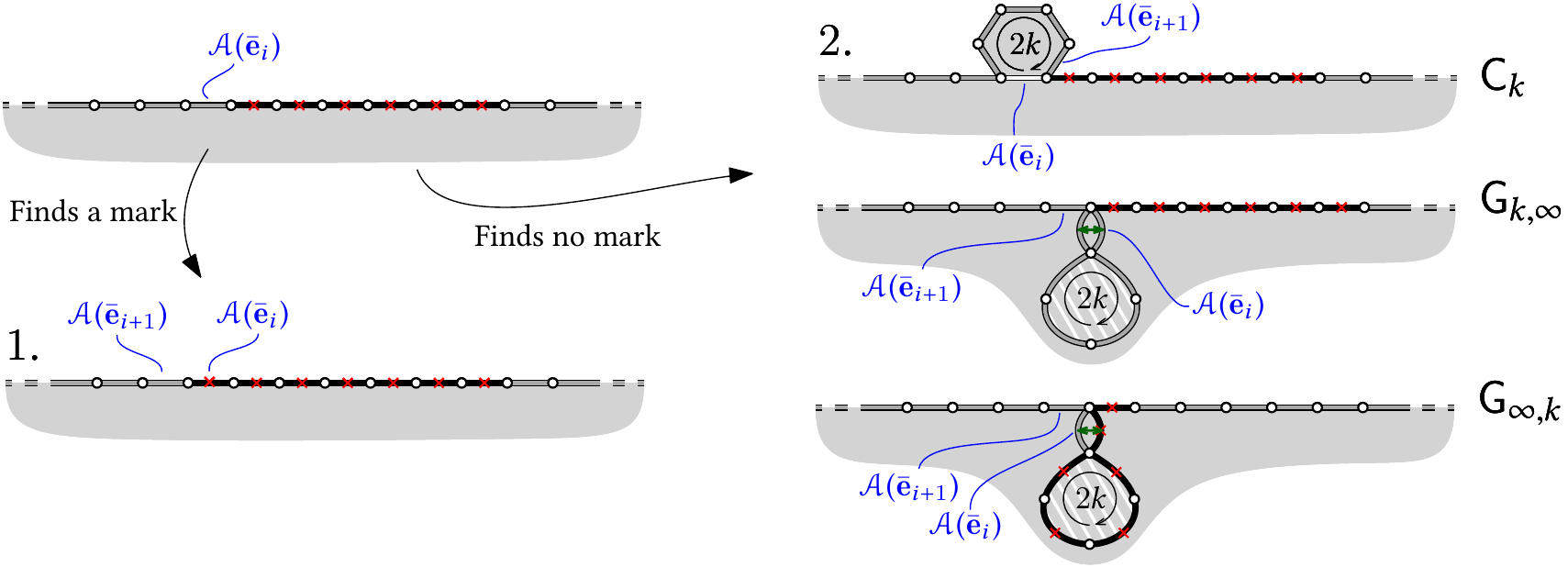}
	\caption{The peeling Algorithm \ref{alg:Peelingperco}. Free edges are represented in darkgray, the explored region in lightgray, and the filled-in regions are hatched. Note that in the case $\Gs_{k,\infty}$, after the gluing operation the resulting edge $\A(\e_i)$ will be white (it carries no mark), while in the case $\Gs_{\infty,k}$ it will be black.}
	\label{fig:PeelingPerco}
	\end{figure}
	
	Here also, some care is needed after an event of type $ \mathsf{G}_{\infty,k}$ where the length of the black boundary is exactly $2k+1$: as in the finite setting, the extremity of the peeled edge then serves as a black boundary of length zero in order to continue the exploration.

In order to describe the peeling exploration of Algorithm \ref{alg:Peelingperco} run on $\Minf$ quantitatively, it is convenient to introduce for every $0\leq i \leq \tau$ the length $B_i$ of the black boundary of the hole of $\e_{i}$ (in terms of number of edges). By definition we have $B_{0}=1$, and $(B_i : 0\leq i \leq \tau)$ evolves as follows: if the exploration of $\A(\e_i)$ reveals a mark, then $ B_{i+1}= B_{i}+1$. If we discover no mark and trigger a peeling step, the only non-zero variation is produced by events of type $ \mathsf{G}_{\infty,k}$ where we then have $B_{i+1} = B_{i} -2k-1$. We choose the convention  that the last quantity may be negative, in which case $\tau=i+1$ and the process stops (moreover, the value $B_\tau$ has no meaning in terms of length of black boundary anymore). Combining the last observation with \eqref{eqn:LawPeeling} we arrive at the following result.

\begin{lemma}\label{lem:LawB} The process $(B_{i} : 0\leq i \leq \tau) $ is a random walk killed at time $\tau = \inf\{ i \geq 0 : B_{i}<0\}$, whose increments are distributed as 
	
	\begin{equation}\label{eqn:LawBstar}
\Delta B:= B_1-B_0= \left\lbrace
\begin{array}{cccc}
1  & \mbox{\footnotesize proba. } & p & \\
-(2k+1) & \mbox{\footnotesize proba. } & \frac{1-p}{2}\nu_\q(-k-1) & (k\geq 0) \\
0 & \mbox{\footnotesize otherwise }
\end{array}\right..
\end{equation}\end{lemma}

An important consequence of these observations is the following. Recall that $\vert \C \vert $ stands for the total number of vertices of the percolation cluster of the origin in $\m_\infty$. Then we have
\begin{equation}\label{eqn:RelationSizeCTau}
	\{\vert \C \vert=\infty\}=\{\tau=\infty\}.
\end{equation} Indeed, while $i<\tau$, the black edges discovered by the peeling algorithm belong to $\C$. Moreover, if $\tau<\infty$, the edge peeled at time $\tau$ is a white edge that encloses $\C$ into a finite region of the map. We can use this observation to compute, as in \cite[Theorem 12]{curien_peeling_2016}, the bond percolation threshold in the $\q$-$\IBHPM$, which is then defined by

\label{sec:half-planepeeling}
\[p_\q^c:=\inf\left\lbrace p\in [0,1] : \Pqinf(\vert \C \vert = \infty)>0\right\rbrace.\]  With this definition, the percolation threshold is annealed, however it is also equal to the quenched threshold, see \cite[Proposition 31]{curien_peeling_2016}. Using Lemma \ref{lem:LawB}, we see that $\tau$ has a positive probability to be infinite if and only if the random variable $\Delta B$ has positive mean, which amounts to 
$$ p -  \frac{1-p}{2}\underbrace{\sum_{k\geq 0} (2k+1) \nu_{ \mathbf{q}}(-k-1)}_{:= \lambda} > 0 \quad \iff \quad p > \frac{\lambda}{\lambda+2}.$$
This shows that $p_{\q}^{c}= \frac{\lambda}{\lambda+2}$, which is equivalent to the expression of \cite[Theorem 12]{curien_peeling_2016} that uses the so-called mean gulp and exposure.

\section{Cut edges in percolation clusters}\label{sec:CutEdges}

The goal is now to prove Theorem \ref{th:dualityperco} in the critical and supercritical cases. Throughout this section, we fix a weight sequence $ \q$ of type $a \in (2,5/2]$, as well as $p\in[p^c_\q, 1]$. We consider the map $\Minf$ and assume as before that its edges are decorated by i.i.d.~numbers of marks with geometric law, so that it corresponds to the bond percolation model with parameter~$p$.\medskip 

  The main idea is to relate (via Proposition \ref{prop:TypeSequences2}) the critical and supercritical cases of Theorem \ref{th:dualityperco} to the probability of the event that the root edge is a cut-edge of the large percolation cluster $\C$, in a sense that we now make precise. For every map $\m$, an edge $e\in\Em$ is a \textit{cut-edge} of $\m$ if and only if $\m\backslash\{e\}$ is not connected. When removing the root edge $e_*$ from the cluster $\C$, we denote by $\C^-$ (resp. $\C^+$) the connected component containing the source (resp. the target) of $\C$ (rooted at the corner defined by $e_*$). We will use the more precise event \begin{equation}\label{eqn:DefCut}
	\cut_k:=\left\lbrace e_* \text{ is a cut-edge of } \C  \text{ and } \#\partial\C^-=2k \right\rbrace, \quad k\in \Z_{\geq 0},
\end{equation} that is illustrated in Figure \ref{fig:Cut}. 

\begin{figure}[h!]
	\centering
	\includegraphics[scale=1.6]{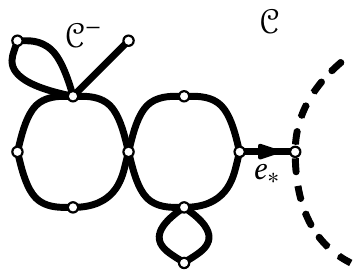}
	\caption{The event $\cut_7$.}
	\label{fig:Cut}
	\end{figure}

\subsection{Relating $ \cut_{k}$ to Theorem \ref{th:dualityperco}}\label{sec:RelateTh1Cut}
From now on, we let $\qt$ stand for the weight sequence defined in \eqref{eqn:DefQt}. The connection between cut-edges and Theorem \ref{th:dualityperco} is established by the following proposition.

\begin{proposition}\label{prop:CutEdgeBoltzmann}
		For every $k\in\Z_{\geq 0}$,
	\[\Pq^{(\infty)}\left(\cut_k \mid \#\partial\C\geq 2m\right) \underset{m\rightarrow \infty}{\longrightarrow} r_{\qt}^{k+1} W_{\qt}^{(k)}.\] 
\end{proposition}

\begin{proof}Let us fix $k\in\Z_{\geq 0}$. We first argue under $\Pq^{(l)}\left(\cdot \mid \#\partial\C\geq 2m\right)$ for fixed $l,m\in\N$ and write
\begin{equation}\label{eqn:SumCut}
	\Pq^{(l)}\left(\cut_k \mid \#\partial\C\geq 2m\right)=\frac{\sum_{j\geq m} \Pq^{(l)}\left(\cut_k \mid \#\partial\C=2j\right)\Pq^{(l)}\left(\#\partial\C=2j\right)}{\sum_{j\geq m} \Pq^{(l)}\left(\#\partial\C=2j\right)}.
\end{equation} Then observe that for every $l\in\N$, by Proposition \ref{prop:LawClusterB},
	\[\Pq^{(l)}\left(\cut_k \mid \#\partial\C=2j\right)=\P^{(j)}_{\qt}(\cut_k).\] In the right-hand side, the event $\cut_k$ is to be interpreted as the fact that the root edge separates the map into two maps with a boundary of respective perimeter $2k$ and $2(j-k-1)$ (formally, all the edges are black under $\P^{(j)}_{\qt}$). By definition of the Boltzmann measure, this yields
	\[\Pq^{(l)}\left(\cut_k \mid \#\partial\C=2j\right)=\frac{W_{\qt}^{(k)}W_{\qt}^{(j-k-1)}}{W_{\qt}^{(j)}}.\] By \cite[Lemma 6]{curien_peeling_2016}, since $\qt$ is an admissible weight sequence, we get
	\begin{equation}\label{eqn:CvgFt}
		\Pq^{(l)}\left(\cut_k \mid \#\partial\C=2j\right)\underset{j \rightarrow \infty}{\longrightarrow} W_{\qt}^{(k)} r_{\qt}^{k+1}.
	\end{equation} Back to \eqref{eqn:SumCut}, we find by Cesàro summation
	\begin{equation}\label{eqn:CvgCutPl}
		\Pq^{(l)}\left(\cut_k \mid \#\partial\C\geq 2m\right)\underset{m \rightarrow \infty}{\longrightarrow} W_{\qt}^{(k)} r_{\qt}^{k+1},
	\end{equation} and the convergence holds uniformly for $l\in\N$ because $\Pq^{(l)}\left(\cut_k \mid \#\partial\C=2j\right)$ does not depend on $l$. Now, the events $\{\#\partial\C < 2m\}$ and $\cut_k$ are both measurable with respect to the ball of radius $2k\vee 2m +1$, so that by the local convergence \eqref{eqn:qIBHPM}, for every $m\in\N$, 
	\[\Pq^{(l)}\left(\cut_k \mid \#\partial\C\geq 2m\right)\underset{l \rightarrow \infty}{\longrightarrow} \Pqinf\left(\cut_k \mid \#\partial\C\geq 2m\right).\] Since \eqref{eqn:CvgCutPl} holds uniformly for $l\in\N$, this concludes the proof.\end{proof}

The second step of the proof of Theorem \ref{th:dualityperco} is to estimate directly $\Pqinf(\cut_k \mid \#\partial\C\geq 2m)$ when $m$ goes to infinity and $k$ is large using the peeling exploration introduced in Section~\ref{sec:half-planepeeling}. In order to avoid meaningless complications, the trivial case $p=1$ is excluded. The main result of this section is the following.

\begin{theorem}\label{thm:CutEdgeIBHPM} Let $ \q$ be a weight sequence of type $a \in (2,5/2]$.
\begin{itemize}
	\item Critical case. If $p=p_\q^c$, we have
\[\lim_{m\rightarrow\infty}\Pqinf(\cut_k\mid \#\partial\C \geq 2m)\underset{k \rightarrow \infty}{\approx} k^{-\frac{a}{a-1}}.\]

\item Supercritical case. If $p_\q^c<p<1$, we have
\[\lim_{m\rightarrow\infty}\Pqinf(\cut_k\mid \#\partial\C \geq 2m)\underset{k \rightarrow \infty}{\approx} k^{-a}.\]
\end{itemize}
\end{theorem}

Note that the limits in Theorem \ref{thm:CutEdgeIBHPM} exist thanks to Proposition \ref{prop:CutEdgeBoltzmann}. The proof of Theorem~\ref{thm:CutEdgeIBHPM} occupies the next sections. Before that, let us show how the proof of Theorem \ref{th:dualityperco} (in the critical and supercritical cases) stems from these two results.

\begin{proof}[Proof of Theorem \ref{th:dualityperco} (Critical and supercritical cases)] We first note that for $p=1$, the result is trivial because the cluster $\C$ is the whole map. Then we apply Proposition \ref{prop:CutEdgeBoltzmann} and Theorem~\ref{thm:CutEdgeIBHPM} to get that \begin{equation*}
	\begin{array}{cccl} 
	r_{\qt}^k W_{\qt}^{(k)} & \underset{k\rightarrow \infty}{\approx} & k^{-\frac{a}{a-1}}  & \mbox{if} \quad p=p_\q^c,  \\
	\\
	r_{\qt}^k W_{\qt}^{(k)}  &\underset{k\rightarrow \infty}{\approx} & k^{-a}  & \mbox{if}  \quad p_\q^c<p<1.  \\
\end{array}\end{equation*} The criteria of Propositions \ref{prop:TypeSequences} and \ref{prop:TypeSequences2} give the expected result, recalling that $a=\alpha+1/2$.\end{proof}

\subsection{$\Hcut_{k}$ and the peeling process}  
In order to prove Theorem \ref{thm:CutEdgeIBHPM}, the main idea is to relate $ \cut_{k}$ to yet another event $\Hcut_{k}$ where we require furthermore than $ \mathcal{C}^{-}$ does not surround $ \mathcal{C}^{+}$. The last event is convenient because it can be evaluated in terms of the peeling process via Algorithm \ref{alg:Peelingperco} (see Lemma \ref{lem:LinkHcutPeel}) and we will show in the next section that for large values of $k$, the events $ \cut_{k}$ and $\Hcut_{k}$ are almost equivalent (see Lemma \ref{lem:CutVSCutP}). 
 
\medskip

Let us first introduce some notation regarding the process introduced in Section \ref{sec:BondPercolation}. Recall that $(B_{i} : 0\leq i \leq \tau)$ measures the black boundary length in the exploration of $\Minf$ driven by Algorithm~\ref{alg:Peelingperco}. Since we have $\Delta B= 0$ with positive probability, it is more convenient to work with the subordinated process $(B^*_i : 0\leq i \leq \tau^*)$ which has non-zero steps. More precisely, we put $\sb_0=0$ and 
\begin{equation}\label{eqn:DefSig}
	\sb_{i+1}=\inf\left\lbrace j> \sb_i : B_{j}\neq B_{j-1} \right\rbrace, \quad \text{ as well as } \quad B^*_{i}=B_{\sb_i},
\end{equation} for every $0\leq i \leq \tau^*:=\inf\{i> 0 : B^*_i < 0\}$. Note that $(B^*_i : 0\leq i \leq \tau^*)$ is still a (killed) random walk, and we let $\Delta B^*:=B^*_1-B^*_0$. In particular, $\tau^*$ can also be considered as the lifetime of the peeling process (in the sense that $\sb_{\tau^*}=\tau$). The first exit time of $\Z_{>0}$ by the process $(B^*_i : 0\leq i \leq \tau^*)$ is denoted by \begin{equation}\label{eqn:DefT}
	T^*:=\inf\{i> 0 : B^*_i \leq 0\},
\end{equation} and plays an important role in the analysis of cut-edges.

We now introduce the event
\begin{equation}\label{eqn:DefCutP}
	\Hcut_k:= \cut_k \cap \left\lbrace \C^-\cap \Br = \emptyset \right\rbrace,\quad k\in \Z_{\geq 0},
\end{equation} where we recall that $\C^-$ is the connected component of $\C\backslash\{e_*\}$ containing the origin vertex, and $\Br$ is the boundary of $\Minf$ on the right of the root edge. The interest of the event $\Hcut$ is that it can be related to the peeling Algorithm \ref{alg:Peelingperco} through the following lemma. 
\begin{lemma}\label{lem:LinkHcutPeel}
	
	For every $k \in \Z_{\geq 0}$, we have
	\[\Pqinf(\Hcut_k)=\frac{1}{1-p}\Pqinf(T^*=2k+1, \ B^*_{T^*}=0).\]
\end{lemma}

\begin{proof} Recall from the Section \ref{sec:IBHPM} that the peeling exploration driven by Algorithm \ref{alg:Peelingperco} follows the left boundary of the percolation cluster $ \mathcal{C}$ until the later intersects the right boundary of the map (at that time, the cluster is swallowed in a finite part of the map and the exploration stops). In this interpretation, each step of the process $(B^{*}_i : 0\leq i \leq \tau^*)$ (i.e.~each time $B_{i+1}-B_i$ is non-zero) corresponds to the exploration of exactly one \emph{half-edge} on the boundary of $ \mathcal{C}$, see Figure~\ref{fig:CutT} for an example. This holds only for $0\leq i < \tau^*$, because the last step of the exploration does not correspond to the discovery of a half-edge on the boundary of $\mathcal{C}$ (but rather on the right boundary of $\Minf$).

Let us examine the situation if $B^{*}_{T^*}=0$.  In that case, the peeling exploration identifies a half-edge to the root edge at time $T^*$, which enforces $N_{e_*}=1$ (otherwise, no peeling step is performed when visiting the half-edge corresponding to the root, since a mark is discovered). This also imposes that $\C^-\cap \Br = \emptyset$, and by the above observation $\#\partial \mathcal{C}^{-} = T^*-1$ so that \[ \left\lbrace T^*=2k+1, \ B^*_{T^*}=0 \right\rbrace \subset \Hcut_k \cap \left\lbrace N_{e_*}=1 \right\rbrace.\] This situation is depicted in Figure \ref{fig:CutT}. Conversely, on the event $\Hcut_k$, the peeling exploration visits a half-edge that is identified to the root edge in $\Minf$. If additionally $N_{e_*}=1$, this half-edge carries no mark, so that a peeling step is performed on it, which enforces $B^*_{T^*}=0$ (and thus $ \#\partial \mathcal{C}^{-} = T^*-1$). This gives
\[\left\lbrace T^*=2k+1, \ B^*_{T^*}=0 \right\rbrace = \Hcut_k \cap \left\lbrace N_{e_*}=1 \right\rbrace.\] Finally, since the event $\Hcut_k$ depends only on the percolated map (and not on the exact number of marks carried by the edges) and since the number of marks $N_{e_*}$ of the root edge is a geometric variable conditioned to be larger than one, we have
\[\Pqinf(\Hcut_k, \ N_{e_*}=1)=(1-p)\Pqinf(\Hcut_k),\] which concludes the proof.  \end{proof}

\begin{figure}[h!]
	\centering
	\includegraphics[scale=1.4]{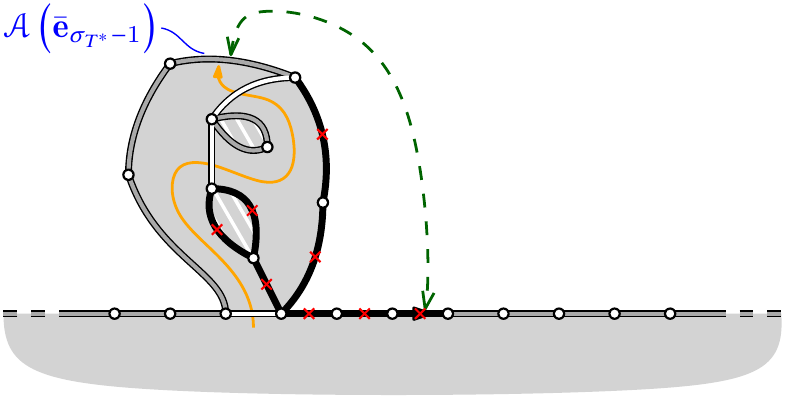}
	\caption{Illustration of the proof of Lemma \ref{lem:LinkHcutPeel}. The peeling process at time $T^*-1$ on the event $\Hcut_k\cap\{N_{e_*}=1\} = \{T^*=2k+1, \ B^*_{T^*}=0\}$. In orange, the percolation interface. In green, the identification of edges performed at time $T^*$ between the root edge and the peeled edge (that carries no mark).}
	\label{fig:CutT}
	\end{figure}

We now move to the estimation of the probability of the event $\Hcut_k$.

\begin{lemma}\label{lem:CutEdgePrimeIBHPM} Let $ \q$ be a weight sequence of type $a \in (2,5/2]$.
\begin{itemize}
	\item Critical case. If $p=p_\q^c$, we have
\[\Pqinf(\Hcut_k)\underset{k \rightarrow \infty}{\approx}k^{-\frac{a}{a-1}}.\]

\item Supercritical case. If $p_\q^c<p<1$, we have
\[\Pqinf(\Hcut_k)\underset{k \rightarrow \infty}{\approx} k^{-a}.\]
\end{itemize}

\end{lemma}

\begin{remark} Later on, we will use the following consequence of Lemma \ref{lem:CutEdgePrimeIBHPM}, that is obtained by Karamata's theorem \cite[Proposition 1.5.10]{bingham_regular_1989}: when $p=p_\q^c$, we have
\begin{equation}\label{eqn:SumCut2}
	\sum_{k\geq m}\Pqinf(\Hcut_k)\underset{m \rightarrow \infty}{\approx} m^{-\frac{1}{a-1}}.
\end{equation}	
\end{remark}

\begin{proof} By Lemma \ref{lem:LinkHcutPeel}, it suffices to estimate $\Pqinf(T^*=2k+1, \ B^*_{T^*}=0)$. For convenience, we consider the random walk $(B^*_i : i\geq 0)$ without killing. Using Lemma \ref{lem:LawB}, we have that $(B^*_i : i\geq 0)$ is centered if $p=p_\q^c$, and has negative (resp. positive) drift if $p<p_\q^c$ (resp. $p>p_\q^c$). We now use Feller's combinatorial (or ``cyclic") lemma \cite[Chapter XII.6]{feller_introduction_1971} to get that
\begin{equation}\label{eqn:CyclicLemma}
	\Pqinf\left(T^*=2k+1, \ B^*_{T^*}=0\right)=\frac{1}{2k+1}\Pqinf\left(B^*_{2k+1}=0\right).
\end{equation} The cases of critical and supercritical percolation are now treated separately.

\smallskip

\noindent\textit{Critical case $(p=p_\q^c)$.} The random walk $(B^*_i : i\geq 0)$ is centered. We apply the local limit theorem (\cite[Theorem 4.2.1]{ibragimov_independent_1971}) to $(B^*_i :  i \geq 0)$, whose steps are in the domain of attraction of a stable distribution with parameter $a-1$. We obtain that
\[\Pqinf(B^*_{k}=0) \underset{k \rightarrow \infty}{\approx} k^{-\frac{1}{a-1}},\] which gives the expected result by \eqref{eqn:CyclicLemma}.
	
\smallskip

\noindent\textit{Supercritical case $(p_\q^c<p<1)$.} The random walk $(B^*_i : i\geq 0)$ has positive drift $d>0$. We introduce the centered version $(\bar{B}_i : i\geq 0)$ of the random walk $(B^*_i : i\geq 0)$  by setting $\bar{B}_i=-B^*_i+di$ for every $i\in\Z_{\geq 0}$ (notice the minus sign). Clearly, we have 
\[\Pqinf\left(B^*_{2k+1}=0\right)=\Pqinf\left(\bar{B}_{2k+1}=d(2k+1)\right).\] Recall that $\Delta B^*$ is in the domain of attraction of a (spectrally negative) stable distribution with parameter $a-1$. In fact, by Lemma \ref{lem:LawB}, the definition \eqref{eqn:DefNu} of $\nu_\q$ and Proposition \ref{prop:TypeSequences2}, we have the more precise identity \begin{equation}\label{eqn:DomainAttrStrong}
	\P\left(\Delta B^*=-(2k+1)\right)\underset{k \rightarrow \infty}{\approx}k^{-a}.
\end{equation} By \cite[Theorem 9.1]{denisov_large_2008}, this ensures that the centered variable $\bar{B}_1$ is $(0,1]$-subexponential, which means that
\[\Pqinf\left(\bar{B}_{k}\in (x,x+1]\right)\underset{k \rightarrow \infty}{\sim} k \cdot \Pqinf\left(\bar{B}_{1}\in (x,x+1]\right),\] uniformly for $x\geq \varepsilon k$, for every $\varepsilon>0$. Thus, we get
\[\Pqinf\left(B^*_{2k+1}=0\right)=\Pqinf\left(\bar{B}_{2k+1}=d(2k+1)\right)\underset{k \rightarrow \infty}{\approx}k^{1-a}.\] The identity \eqref{eqn:CyclicLemma} concludes the proof. \end{proof}

\subsection{$\cut_{k}$, $\Hcut$ and the proof of Theorem \ref{thm:CutEdgeIBHPM}}
We finally prove Theorem \ref{thm:CutEdgeIBHPM} from Lemma \ref{lem:CutEdgePrimeIBHPM}. This will complete the proof of Theorem \ref{th:dualityperco} in the supercritical and critical cases. For this we will show that $\cut_{k}$ and $\Hcut_{k}$ are almost equivalent for large values of $k$.
The first step is to establish some bounds for the tail distribution of the perimeter of the percolation cluster under $\Pqinf$.

\begin{lemma}\label{lem:TailSizeCluster}Let $ \q$ be a weight sequence of type $a \in (2,5/2]$.\begin{itemize}
	\item Critical case. If $p=p_\q^c$, there exists slowly varying functions $L_1$ and $L_2$ such that for every $m\in\N$,
\[L_1(m) m^{-\frac{a-2}{a-1}}\leq \Pqinf\left(\#\partial\C \geq 2m \right)\leq L_2(m) m^{-\frac{a-2}{a-1}}.\]

\item Supercritical case. If $p_\q^c<p<1$, we have
\[\Pqinf\left(\#\partial\C \geq 2m \right)\underset{m \rightarrow \infty}{\longrightarrow} \Pqinf\left(\#\partial\C = \infty \right) > 0.\]\end{itemize}

\end{lemma}

\begin{proof} The proof strategy closely follows that of \cite[Theorem 2]{angel_percolations_2015}. The main issue is that the peeling Algorithm \ref{alg:Peelingperco} does not provide the perimeter $\#\partial\C$ of the cluster. Indeed, the finite hole of the map $\e_{\tau}$ is not visited by the peeling exploration and possibly contains part of~$\partial\C$ (see Figure \ref{fig:PeelingT} for an example). We circumvent this problem by defining a \textit{right peeling exploration}, symmetric to that of Algorithm \ref{alg:Peelingperco} in the sense that it always explores the free edge on the right of the black boundary of the hole. We denote by $(\tilde{B}_i : i \geq 0)$ the process defined as $(B^*_i : i\geq 0)$ for the right peeling exploration (both unkilled for convenience), and let $\tilde{\tau}:=\inf\{i\geq 0 : \tilde{B}_i< 0\}$. An argument similar to the proof of Lemma \ref{lem:LinkHcutPeel} then shows that  \[\tau^* \leq\#\partial\C \leq \tau^* + \tilde{\tau}.\]
Since $\tau^*$ and $\tilde{\tau}$ have the same distribution (though not independent) we get
	\[\Pqinf\left(\tau^* \geq 2m \right) \leq \Pqinf\left(\#\partial\C \geq 2m \right)\leq 2\Pqinf\left(\tau^* \geq m \right).\] We now treat the cases of critical and supercritical percolation separately.
	
\smallskip

\noindent\textit{Critical case $(p=p_\q^c)$.} The random walk $(B^*_i : i\geq 0)$ is centered. We use \cite[Theorem 1]{doney_exact_1982}, which ensures that \[\Pqinf\left(\tau^* \geq m \right)\underset{m \rightarrow \infty}{\approx} m^{\frac{1}{a-1}-1},\] since $\Delta B^*$ is in the domain of attraction of a stable distribution with parameter $a-1$. (To be complete, we also need \cite[Theorem 8.2.18]{borovkov_asymptotic_2008} because $(B^*_i : i\geq 0)$ starts at $B^*_0=1$.)

\smallskip

\noindent\textit{Supercritical case $(p_\q^c<p<1)$.} This case immediately follows by the definition of $p^c_\q$ and the continuity of probability along a monotone sequence of events.\end{proof}

\begin{figure}[ht]
	\centering	
	\includegraphics[scale=1.4]{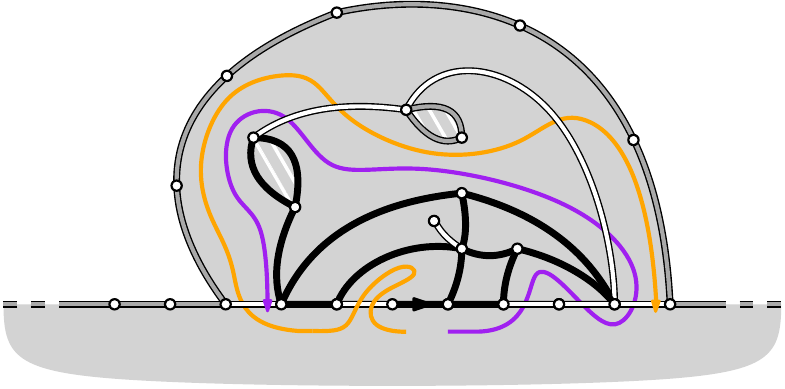}
	\caption{The left percolation interface of the cluster $\C$ until $\tau^*$ (in orange), and its right percolation interface until $\tilde{\tau}$ (in purple). }
	\label{fig:PeelingT}
	\end{figure}

The next lemma establishes a sensitivity property of $\Pqinf(\#\partial\C\geq 2m)$ with respect to the boundary condition. Recall that $\Pqinf$ is the law of the $\q$-$\IBHPM$ equipped with marks corresponding to a bond percolation model with ``free-black-free" boundary condition. We denote by $\Ptinf$ the law of the same model with totally free boundary condition, i.e.~where all the edges carry i.i.d.\ numbers of marks with geometric law of parameter $p$. Recall that the percolation cluster $\C$ is the black connected component of the \textit{origin} (the source of the root edge). 

\begin{lemma}\label{lem:BdaryCondition}Let $ \q$ be a weight sequence of type $a \in (2,5/2]$. Then, we have
	\[\frac{\Pt_\q^{(\infty)}(\#\partial\C\geq 2m)}{\Pqinf(\#\partial\C\geq 2m)}\underset{m \rightarrow \infty}{\longrightarrow} c\in(0,1).\]
\end{lemma}

\begin{remark}At first glance, one could expect the constant $c$ in the statement of Lemma \ref{lem:BdaryCondition} to be equal to one, meaning that the asymptotics of $\Pqinf(\#\partial\C\geq 2m)$ do not depend on the boundary condition. However, as the proof should reveal, the black root edge allows two ``seeds" for the percolation cluster, which changes the asymptotics of $\Pqinf(\#\partial\C\geq 2m)$ by a constant factor.
\end{remark}

\begin{proof} We argue under $\Pt_\q^{(\infty)}$, and reveal the status of the free root edge in the sense of Algorithm \ref{alg:Peelingperco}. Note that under our assumptions, this root (half-)edge is black if and only if it carries a cross, because there is no other black edge on the boundary it can be identified to. As usual, two cases may to happen.

\begin{itemize}
	\item The root edge carries a mark. Then the map has ``free-black-free" boundary condition as for $\Pqinf$.
	\item The root edge carries no mark, and a peeling step is performed. If the peeling step is of type $\Cs$ or $\Gs_{\infty,\cdot}$, the event $\{\#\partial\C \geq 2m\}$ is realized if and only if it is realized in the infinite hole of the map, that has law $\Pt_\q^{(\infty)}$ by the spatial Markov property.
\end{itemize} Finally, we let $\Gs^{(l)}_\infty$ be the event that the root edge is white and identified to an edge on its left. Then we have
\begin{multline*}
	\Pt_\q^{(\infty)}(\#\partial\C\geq 2m)=p\Pqinf(\#\partial\C\geq 2m)+(1-p)\left( \sum_{k\geq 1} \frac{\nu_\q(-k)}{2} \right) \Pt_\q^{(\infty)}(\#\partial\C\geq 2m) \\
	+ \Pt_\q^{(\infty)}\left(\#\partial\C\geq 2m, \ \Gs^{(l)}_\infty\right),
\end{multline*} from where we obtain
\begin{equation}\label{eqn:DecompositionFirstPeeling}
	\frac{\Pt_\q^{(\infty)}(\#\partial\C\geq 2m)}{\Pqinf(\#\partial\C\geq 2m)}=p+(1-p) \left( \sum_{k\geq 1} \frac{\nu_\q(-k)}{2} \right) \frac{\Pt_\q^{(\infty)}(\#\partial\C\geq 2m)}{\Pqinf(\#\partial\C\geq 2m)}+\frac{\Pt_\q^{(\infty)}\left(\#\partial\C\geq 2m,\ \Gs^{(l)}_\infty\right)}{\Pqinf(\#\partial\C\geq 2m)}.
\end{equation} To conclude, we will show that
\begin{equation}\label{eqn:CvgZeroLastTerm}
	\frac{\Pt_\q^{(\infty)}\left(\#\partial\C\geq 2m, \ \Gs^{(l)}_\infty\right)}{\Pqinf(\#\partial\C\geq 2m)}\underset{m \rightarrow \infty}{\longrightarrow} 0.
\end{equation} Intuitively, this comes from the fact that on the event $\Gs^{(l)}_\infty$, the percolation cluster $\C$ is confined in a finite hole of the map. More precisely, the key observation is the following almost sure inclusion of events:
\begin{equation}\label{eqn:InclusionEvents}
	\left\lbrace\#\partial\C\geq 2m, \ \Gs^{(l)}_\infty\right\rbrace \subset \bigcup_{j\geq m}{\Hcut_j}.
\end{equation} Indeed, on the event $\Gs^{(l)}_\infty$, the white root edge separates the map into a finite hole containing $\C$ and an infinite hole (see Figure \ref{fig:CutClosed}). As seen in the proof of Lemma \ref{lem:LinkHcutPeel}, since the root edge carries no mark, in such situation we have $\#\partial\C^-=T^*-1$, and moreover $\partial\C=\partial\C^-$ which yields \eqref{eqn:InclusionEvents} by definition of $\Hcut_j$. (Note that since $\#\partial\C^-=T^*-1$, the stopping time $T^*$ is necessarily odd because the map is bipartite.) We can now prove assertion \eqref{eqn:CvgZeroLastTerm}.

\smallskip

\noindent\textit{Critical case $(p=p_\q^c)$.} By Lemma \ref{lem:CutEdgePrimeIBHPM} and \eqref{eqn:SumCut2}, we get that \[\Pt_\q^{(\infty)}\left(\#\partial\C\geq 2m, \ \Gs^{(l)}_\infty\right)\leq \sum_{j\geq m}\Pqinf(\Hcut_j)\underset{m \rightarrow \infty}{\approx} m^{-\frac{1}{a-1}},\] while by Lemma \ref{lem:TailSizeCluster}, we have $\Pqinf(\#\partial\C\geq 2m)\geq L_1(m)m^{-\frac{a-2}{a-1}}$. This yields the expected result, since $a\in(2,5/2]$.

\smallskip

\noindent\textit{Supercritical case $(p_\q^c<p<1)$.} By Lemma \ref{lem:CutEdgePrimeIBHPM} and since $a>2$, we know that  
\[\Pt_\q^{(\infty)}\left(\#\partial\C\geq 2m, \ \Gs^{(l)}_\infty\right) \leq \sum_{j\geq m}\Pqinf(\Hcut_j)\underset{m \rightarrow \infty}{\longrightarrow} 0\] as the rest of a convergent series. This  is enough to conclude because $\Pqinf(\#\partial\C\geq 2m) \rightarrow \Pqinf(\#\partial\C = \infty)>0$ as $m$ goes to infinity. 

\end{proof}

\begin{figure}[ht]
	\centering
	\includegraphics[scale=1.6]{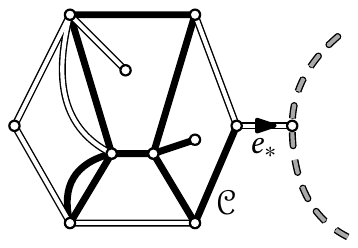}
	\caption{The percolation cluster $\C$ on the event $\Gs^{(l)}_\infty$.}
	\label{fig:CutClosed}
	\end{figure}
	
The last preliminary lemma shows that the events $\cut_k$ and $\Hcut_k$ have the same probability under $\Pqinf\left(\cdot \mid \#\partial\C \geq 2m \right)$ when $m$ tends to infinity.

\begin{lemma}\label{lem:CutVSCutP} For every $k\in\Z_{\geq 0}$,
	\[\lim_{m \rightarrow \infty}\left| \Pqinf(\cut_k\mid \#\partial\C \geq 2m) - \Pqinf(\Hcut_k\mid \#\partial\C \geq 2m) \right| =0.\]
\end{lemma}

\begin{proof} By definition \eqref{eqn:DefCutP} of the event $\Hcut_k$, it suffices to show that for every $k\in\Z_{\geq 0}$,\begin{equation}\label{eqn:LimitProba1}
	\Pqinf\left(\cut_k, \ \C^-\cap \Br \neq \emptyset \mid \#\partial\C \geq 2m\right)\underset{m \rightarrow \infty}{\longrightarrow} 0.
\end{equation} Let $k\in\Z_{\geq 0}$, and define the symmetric events of $\cut_k$ and $\Hcut_k$ for the right part $\C^+$ of the percolation cluster,
\[\cut'_k=\left\lbrace e_* \text{ is a cut-edge of } \C \text{ and } \partial\C^+=2k\right\rbrace \quad \text{and} \quad \Hcut'_k:=\cut'_k\cap \left\lbrace \C^+\cap \Bl = \emptyset \right\rbrace,\] where $\Bl$ is the boundary of $\Minf$ on the left of the root edge. Then we have the almost sure inclusion of events
\begin{equation}\label{eqn:InclusionCutR}
	\cut_k\cap\left\lbrace \C^-\cap \Br \neq \emptyset, \ \#\partial\C\geq 2m\right\rbrace \subset \bigcup_{j \geq m-k-1}\Hcut'_j.
\end{equation} Indeed, when $\C^-\cap \Br \neq \emptyset$, the left part $\C^-$ of the cluster hits the right boundary of the map $\Minf$. If additionally the event $\cut_k$ is realized, $\C^+$ cannot intersect the left boundary of $\Minf$ (otherwise, the root edge would not be a cut-edge of $\C$). Finally, on the event $\cut_k$, we have $\#\partial\C^-=2k$ so that if $\#\partial\C \geq 2m$, then $\#\partial\C^+\geq 2(m-k-1)$. This concludes the proof of \eqref{eqn:InclusionCutR}, which is illustrated in Figure \ref{fig:PeelingTBis}. Note that by symmetry, the events $\Hcut_j$ and $\Hcut'_j$ have the same probability for every $j \in \Z_{\geq 0}$.

From Equation \eqref{eqn:InclusionCutR} we obtain
\begin{equation}\label{eqn:EndLemmaSupercritOrCrit}
	\Pqinf(\cut_k, \ B_T<0 \mid \#\partial\C \geq 2m)\leq \frac{\sum_{j\geq m-k-1}{\Pqinf\left(\Hcut_{j} \right)}}{\Pqinf(\#\partial\C \geq 2m)}.
\end{equation} We now argue as in the end of Lemma \ref{lem:BdaryCondition} to conclude the proof. \end{proof}

\begin{figure}[ht]
	\centering	
	\includegraphics[scale=1.4]{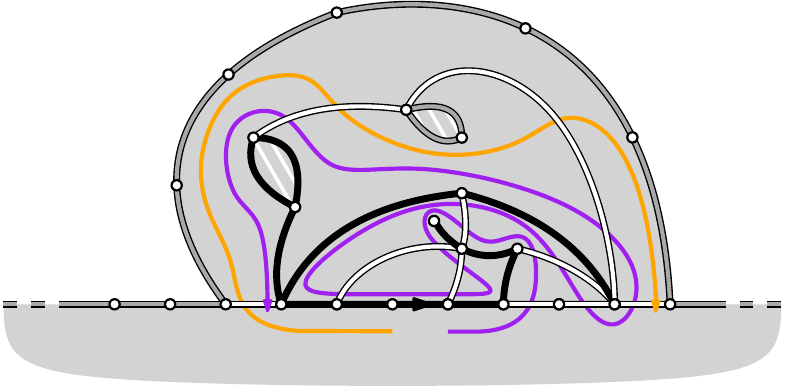}
	\caption{The map $\Minf$ on the event $\cut_k\cap\{\C^-\cap \Br \neq \emptyset, \ \#\partial\C\geq 2m\}$, together with (part of) the left and right percolation interfaces. Note that the event $\Hcut'_j$ has to be realized for some $j\geq m-k-1$.}
	\label{fig:PeelingTBis}
	\end{figure}

We can finally prove Theorem \ref{thm:CutEdgeIBHPM}.

\begin{proof}[Proof of Theorem \ref{thm:CutEdgeIBHPM}] By Lemma \ref{lem:CutVSCutP}, it suffices to find an equivalent as $k$ goes to infinity of the quantity
\[\lim_{m\rightarrow \infty}\Pqinf(\Hcut_k\mid \#\partial\C \geq 2m).\] Moreover, by arguing as in Lemma \ref{lem:LinkHcutPeel}, we have
\begin{equation}\label{eqn:HcutCond}
	\Pqinf(\Hcut_k\mid \#\partial\C \geq 2m)=\frac{1}{1-p}\Pqinf(T^*=2k+1,\ B_T=0\mid \#\partial\C \geq 2m), \quad k,m\in\Z_{\geq 0}.
\end{equation} Let $k\in\Z_{\geq 0}$, and recall that $\sb_{T^*}$ is a stopping time with respect to the filtration of the peeling process, so that the infinite hole of the map $\e_{\sb_{T^*}}$ has the law $\Pt^{(\infty)}_\q$ of the $\q$-$\IBHPM$ with free boundary condition. Moreover, on the event $\{T^*=2k+1, \ B^*_{T^*}=0\}$, $\C^+$ is the cluster of the target of the root edge in the infinite hole of $\e_{\sb_{T^*}}$, and $\#\partial\C \geq 2m$ if and only if $\#\partial\C^+\geq 2(m-k-1)$ (see Figure \ref{fig:PeelingT4} for an illustration). By applying the (spatial) Markov property at time $T^*$, we get
\[\Pqinf\left(T^*=2k+1,\ B^*_{T^*}=0, \ \#\partial\C \geq 2m\right)=\Pqinf\left(T^*=2k+1,\ B^*_{T^*}=0\right)\Pt^{(\infty)}_\q(\#\partial\C\geq 2(m-k-1)).\] From Lemma \ref{lem:BdaryCondition}, \eqref{eqn:HcutCond} and Lemma \ref{lem:LinkHcutPeel}, we can rewrite this as
\begin{equation}\label{eqn:LimitProba2}
	\lim_{m\rightarrow \infty}\Pqinf\left(\Hcut_k \mid \#\partial\C \geq 2m\right) = c\cdot \Pqinf\left(\Hcut_k \right)\lim_{m\rightarrow \infty} \frac{\Pqinf(\#\partial\C \geq 2(m-k-1))}{\Pqinf(\#\partial\C \geq 2m)},
\end{equation} where $c\in(0,1)$ is the constant in the statement of Lemma \ref{lem:BdaryCondition}. Moreover, by Proposition \ref{prop:CutEdgeBoltzmann} and Lemma \ref{lem:CutVSCutP}, we know that the limit on the left-hand side exists. Let us define 
\[\kappa := \lim_{m\rightarrow \infty} \frac{\Pqinf(\#\partial\C \geq 2(m-1))}{\Pqinf(\#\partial\C \geq 2m)}.\] Then the power series whose coefficients are $(\Pqinf(\#\partial\C \geq 2m) : m\geq 0)$ has radius of convergence $\kappa$. By Lemma \ref{lem:TailSizeCluster}, we deduce that for $p_\q^c\leq p<1$, we have $\kappa=1$. Together with Equation \eqref{eqn:LimitProba2} and Lemma \ref{lem:CutEdgePrimeIBHPM}, this concludes the proof.\end{proof}

\begin{figure}[ht]
	\centering
	\includegraphics[scale=1.5]{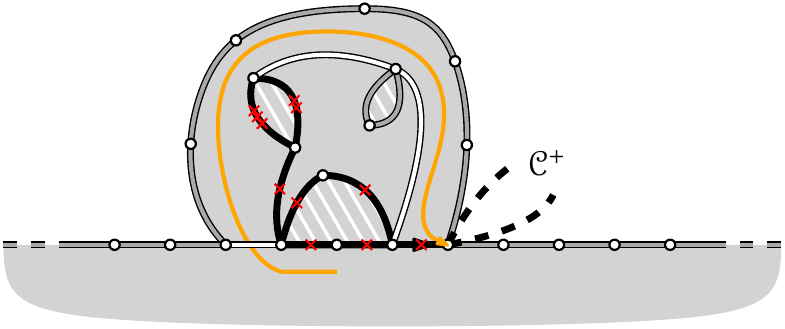}
	\caption{The peeling exploration at time $T^*$, on the event $\{T^*=2k+1, \ B^*_{T^*}=0\}$. The left percolation interface is in orange. The right part $\C^+$ of the cluster is also the cluster of the target of the root edge in the infinite hole of the map $\e_{\sb_{T^*}}$.}
	\label{fig:PeelingT4}
	\end{figure}
	
\begin{remark} The strategy we developed in this section can also be adapted to the case of subcritical percolation. In this setting, one can prove the existence of a constant $C>0$ such that
\[\lim_{m\rightarrow\infty}\Pqinf(\cut_k\mid \#\partial\C \geq 2m)\underset{k \rightarrow \infty}{\sim} C k^{-\frac{3}{2}},\] and deduce Theorem \ref{th:dualityperco} also in this regime. However, this does not provide the sharpness of the phase transition (Proposition \ref{prop:Sharpness}), so that we use another strategy in Section \ref{sec:Subcritical}.
\end{remark}	
	
\section{Subcritical percolation and sharpness of the phase transition}\label{sec:Subcritical}

We now deal with Bernoulli bond percolation in the subcritical regime $(0\leq p< p_\q^c$). As we mentioned, the proof of Theorem \ref{th:dualityperco} in this case can also be obtained by sharpening the study of cut-edges that we made in Section \ref{sec:CutEdges}. However, this would not be sufficient to prove the exponential decay of the size of the percolation cluster claimed in Proposition \ref{prop:Sharpness} and we preferred coming back to the finite setting of $\q$-Boltzmann maps.\\
Throughout this section, we fix a weight sequence $ \q$ of type $a \in (2,5/2]$, as well as $p \in [0,p^c_\q)$. We assume that the $\q$-Boltzmann map $M_{ \q}$ is decorated by i.i.d.~numbers of marks with geometric distribution of parameter $p$, so that it corresponds to the bond percolation model on $M_\q$ with the same parameter. We first establish in Lemma \ref{lem:ExponentialTailPerimeter} an estimate on the tail distribution of the perimeter of the percolation cluster in $ M_{\q}$, from where we deduce Theorem \ref{th:dualityperco} and Proposition \ref{prop:Sharpness} with short proofs.

\begin{lemma}\label{lem:ExponentialTailPerimeter} There exists $C_1,C_2>0$ such that for every $m \in \Z_{\geq 0}$,
\[\Pq(\#\partial \C \geq 2m)\leq C_1\exp(-C_2 m).\]
	
\end{lemma}

\begin{proof}[Proof of Lemma \ref{lem:ExponentialTailPerimeter}] The case $p=0$ is trivial so we assume that $0<p<p_\q^c$. We proceed to the bond percolation exploration of $M_{\q}$ driven by Algorithm \ref{alg:PeelingpercoFinite}. By the estimate \eqref{eqn:BoundPerimeterCLuster}, it suffices to prove that the stopping time $\theta$ of the exploration has an exponential tail. Recall from Lemma \ref{lem:LawBFinite} the notation $B_i$ and $F_i$ for the length of the black and free boundaries at step $i$ of the exploration, for every $0\leq i \leq \theta$. 

On the one hand, we see that for every $A>0$, there exists $ \varepsilon >0$ such that 
 \begin{eqnarray} \label{eq:expo1} \Pq( \theta > i+2 \mid \theta >i, \ B_{i} < A) \leq 1- \varepsilon.  \end{eqnarray}  Indeed, when the length of the free boundary is even, the peeling exploration may stop at the next step by identifying the two free edges neighboring the black boundary. If the length of the free boundary is odd, we may require that we first turn the edge $\A(\e_{i+1})$ into a black edge and then repeat the above argument. Moreover, if the black boundary has bounded length, these events occur with probability bounded away from zero from the explicit transition probabilities $\p$ described in Section \ref{sec:PeelingqBoltzmann} together with \cite[Lemma 6]{curien_peeling_2016}. This reasoning is illustrated by Figure~\ref{fig:PeelingPercoFiniteStop}.
  
On the other hand, we can choose $A$ large enough so that when $B_{i}\geq A$ and $\theta >i$, then $(B_i : 0\leq i \leq \theta)$ has a ``uniform negative drift''. More precisely, recall from Lemma \ref{lem:LawBFinite} that conditionally on $\theta >i$ (i.e., $F_{i}>0$) we have 
\begin{equation*}\label{eqn:LawBstarFinite}
	\Pq(B_{i+1}=B_i+1)= p \quad \text{and} \quad \Pq(B_{i+1}=B_i-2k-1) = (1-p) \p^{(P_i)}(P_i-k-1,k), \quad 0\leq k \leq \frac{1}{2}(B_i-1),
\end{equation*} where $P_i=\tfrac{1}{2}(B_i+F_i)$, and using \cite[Lemma 6]{curien_peeling_2016} again we have
\begin{equation}\label{eqn:ConvProbTrans}
	\p^{(l)}(l+k-1,k)=\displaystyle\frac{W_\q^{(k)}W_\q^{(l-k-1)}}{W_\q^{(l)}}\underset{l \rightarrow \infty}{\longrightarrow} W_\q^{(k)} \rq^{k+1}=\frac{1}{2} \nu_\q(-k-1), \quad k \geq 0.
\end{equation} Since $p < p_{ \q}^{c}$, this ensures (see the end of Section \ref{sec:half-planepeeling}) that \[p -  \frac{1-p}{2}\sum_{k\geq 0} (2k+1) \nu_{ \mathbf{q}}(-k-1)< 0, \quad \text{and so} \quad p -  \frac{1-p}{2}\sum_{k= 0}^{D} (2k+1) \nu_{ \mathbf{q}}(-k-1)< 0\] for $D$ large enough. Using this observation together with \eqref{eqn:ConvProbTrans} and the fact that $P_i \geq B_i/2$, one can find a probability measure $\zeta$ supported on $\{-D, \ldots, 1\}$ with negative mean, such that if $A$ is large enough, then the conditional distribution of $B_{i+1}-B_i$ given $B_{i}\geq A$ and $\theta >i$ is stochastically dominated by $\zeta$. In symbols,
 \begin{eqnarray} \label{eq:expo2}\textup{Law}( B_{i+1}-B_i \mid  \theta >i, \ B_{i}\geq A ) \overset{ \mathrm{sto}}{\leq } \zeta.  \end{eqnarray}
 It is now straightforward to combine \eqref{eq:expo1} and \eqref{eq:expo2} to establish that $\theta$ has an exponential tail. This can be proved along the following lines: from \eqref{eq:expo2} and the fact that $\E[\zeta]<0$, we see that the return times of $(B_i : 0 \leq i \leq \theta)$ to the interval $[0, A)$ have an exponential tail. Hence, using \eqref{eq:expo1} we can bound $\theta$ from above by a sum of a geometric number of i.i.d.\ random variables having exponential tails, so that $\theta$ also has an exponential tail. We leave the details to the reader. \end{proof}

We now prove Theorem \ref{th:dualityperco} in the subcritical case using Lemma \ref{lem:ExponentialTailPerimeter}.

\begin{proof}[Proof of Theorem \ref{th:dualityperco} (Subcritical case)] By Proposition \ref{prop:LawCluster}, the percolation cluster $\C$ of $M_\q$ is a $\qt$-Boltzmann map for a certain admissible weight sequence $\qt$. Let us proceed by contradiction and assume that $\qt$ is critical. By \cite[Proposition 2]{budd_peeling_2016}, the partition function for pointed maps with a boundary of perimeter $2k$ reads
	\[W_{\qt}^{(k,\bullet)}:= \frac{1}{\qt_k}\sum_{ \substack{\m \in \M^\bullet \\ \#\partial\m =2k} }{w_{\qt}(\m)}= (4r_{\qt})^{-k} {2k \choose k}, \quad k\in \Z_{\geq 0}.\] Then, recalling from \eqref{eqn:fq} the definition of the function $f_{\qt}$ and the identity $r_{\qt}=(4Z_{\qt})^{-1}$, we find
\[f_{\qt}(x)=2r_{\qt} \sum_{k= 1}^\infty \left(\frac{x}{Z_{\qt}}\right)^{k-1} \qt_k W_{\qt}^{(k,\bullet)}, \quad x\geq 0.\] Now, by combining \cite[Lemma 6]{curien_peeling_2016} and \cite[Proposition 2]{budd_peeling_2016}, we have that
\[ \lim_{k \rightarrow \infty} \left(W_{\qt}^{(k,\bullet)}\right)^{1/k}=\lim_{k \rightarrow \infty} \left(W_{\qt}^{(k)}\right)^{1/k}=\frac{1}{r_{\qt}}.\] Moreover, since $\C$ is a $\qt$-Boltzmann map and by Lemma \ref{lem:ExponentialTailPerimeter},
\[\Pq(\#\partial\C = 2k)= \P_{\qt}(\#\partial M_{\qt}=2k) \propto \qt_k W^{(k)}_{\qt} \quad \text{and} \quad \Pq(\#\partial\C = 2k) \leq C_1 \exp (-C_2 k), \quad k \in \N,\] for positive constants $C_1$ and $C_2$. This enforces that there exists $\varepsilon >0$ such that $f_{\qt}(Z_{\qt}+\varepsilon)<\infty$. In the terminology of Boltzmann maps, the weight sequence $\qt$ is said to be \textit{regular critical}, see \cite{marckert_invariance_2007}, which is a sub-case of generic critical sequences. In particular, by Theorem \ref{th:dualityperco} (in the critical case), there exists $p^c_{\qt} \in (0,1)$ such that if we perform a bond percolation with parameter $p^c_{\qt}$ on $\C$, the cluster of the origin $\C'$ is a discrete stable map with parameter $7/6$. We let $\q'$ be the weight sequence such that $\C'$ is a $\q'$-Boltzmann map. Then, we get
\begin{equation}\label{eqn:PerimCLuster}
	\P_{\qt}\left( \#\partial \C' \geq 2k \right)=\P_{\q'}(\#\partial M_{\q'}\geq 2k) \propto \sum_{j \geq k}\q'_j W^{(j)}_{\q'},
\end{equation} and the right-hand side decays at least polynomially by Definitions \ref{def:JSLaw} and \ref{def:Classification} as well as Proposition \ref{prop:TypeSequences2}. We finally observe that by a standard coupling argument, the cluster $\C'$ can also be obtained from the initial map $M_\q$ by performing a bond percolation with parameter $p'=p\cdot p^c_{\qt}<p<p^c_\q$. By Lemma \ref{lem:ExponentialTailPerimeter}, this entails
\[\Pq\left( \#\partial\C' \geq 2k \right)\leq C'_1 \exp (-C'_2 k)\] for positive constants $C'_1$ and $C'_2$, in contradiction with \eqref{eqn:PerimCLuster}. This concludes the proof. \end{proof}

We now prove Proposition \ref{prop:Sharpness} in the subcritical case using Lemma \ref{lem:ExponentialTailPerimeter} and Theorem \ref{th:dualityperco}.

\begin{proof}[Proof of Proposition \ref{prop:Sharpness} (Subcritical case)] This proof is inspired by \cite[Proposition 5.1]{bernardi_boltzmann_2017}. By Theorem \ref{th:dualityperco}, the percolation cluster $\C$ of $M_\q$ is a $\qt$-Boltzmann map for a certain subcritical weight sequence $\qt$. We now let the weight sequence $\qt$ vary by defining $\qt(u):=(u^{k-1}q_k : k \in \N)$. Then, we see from \eqref{eqn:fq} that $f_{\qt(u)}(x)=f_{\qt}(ux)$ for every $x \geq 0$. By Lemma \ref{lem:ExponentialTailPerimeter} and using the same argument as in the proof of Theorem \ref{th:dualityperco} (in the subcritical case), we obtain that there exists $\varepsilon>0$ such that $f_{\qt}(Z_{\qt}+\varepsilon)<\infty$. Now, we argue as in \cite[Proposition 5.1]{bernardi_boltzmann_2017}: since $\qt$ is subcritical, we can solve the equation
\[f_{\qt(u)}(x)=1-\frac{1}{x}\] in an open neighborhood of $u=1$ by the implicit function theorem, so that $\qt(u)$ is admissible for some $u>1$ by the criterion recalled in Section \ref{sec:BoltzmannDistributions}. As a consequence, we have
\[ w_{\qt(u)}(\M) := \sum_{\m \in \M} \prod_{f \in \Fm}\qt_{\frac{\deg(f)}{2}} u^{\frac{\deg(f)}{2}-1}= \sum_{\m \in \M} u^{\#\Vm-2} w_{\qt}(\m) <\infty, \] where we used Euler's formula. This entails that
\[\E_\q\left[u^{\vert \C \vert}\right]=\E_{\qt}\left[u^{\#\mathrm{V}(M_{\qt})}\right]=\sum_{\m \in \M} u^{\#\Vm} w_{\qt}(\m) <\infty \] for some $u>1$, as wanted.\end{proof}

\smallskip

\bibliographystyle{siam}
\bibliography{dualityperco}

\begin{thebibliography}{10}

\bibitem{angel_growth_2003}
{\sc O.~Angel}, {\em Growth and percolation on the uniform infinite planar
  triangulation}, Geom. funct. anal., 13 (2003), pp.~935--974.

\bibitem{angel_scaling_2004}
\leavevmode\vrule height 2pt depth -1.6pt width 23pt, {\em Scaling of
  {Percolation} on {Infinite} {Planar} {Maps}, {I}}, arXiv:math/0501006,
  (2004).

\bibitem{angel_percolations_2015}
{\sc O.~Angel and N.~Curien}, {\em Percolations on random maps {I}:
  {Half}-plane models}, Ann. Inst. H. Poincaré Probab. Statist., 51 (2015),
  pp.~405--431.

\bibitem{bernardi_boltzmann_2017}
{\sc O.~Bernardi, N.~Curien, and G.~Miermont}, {\em A {Boltzmann} approach to
  percolation on random triangulations}, Canad. J. Math, (to appear) (2017).
\newblock arXiv:1705.04064[math].

\bibitem{bertoin_martingales_2016}
{\sc J.~Bertoin, T.~Budd, N.~Curien, and I.~Kortchemski}, {\em Martingales in
  self-similar growth-fragmentations and their connections with random planar
  maps}, Probab. Th. Rel. Fields, (to appear) (2016).
\newblock arXiv:1605.00581 [math-ph].

\bibitem{bingham_regular_1989}
{\sc N.~H. Bingham, C.~M. Goldie, and J.~L. Teugels}, {\em Regular
  {Variation}}, Cambridge University Press, July 1989.

\bibitem{bjornberg_site_2015}
{\sc J.~E. Björnberg and S.~O. Stefánsson}, {\em On {Site} {Percolation} in
  {Random} {Quadrangulations} of the {Half}-{Plane}}, J Stat Phys, 160 (2015),
  pp.~336--356.

\bibitem{borot_recursive_2012}
{\sc G.~Borot, J.~Bouttier, and E.~Guitter}, {\em A recursive approach to the
  {O}(n) model on random maps via nested loops}, J. Phys. A: Math. Theor., 45
  (2012), p.~045002.

\bibitem{borovkov_asymptotic_2008}
{\sc A.~A. Borovkov and K.~A. Borovkov}, {\em Asymptotic {Analysis} of {Random}
  {Walks}: {Heavy}-{Tailed} {Distributions}}, Cambridge University Press, June
  2008.

\bibitem{bouttier_planar_2004}
{\sc J.~Bouttier, P.~Di~Francesco, and E.~Guitter}, {\em Planar maps as labeled
  mobiles}, Electron. J. Combin., 11 (2004), pp.~R69, 27 p.

\bibitem{budd_peeling_2016}
{\sc T.~Budd}, {\em The {Peeling} {Process} of {Infinite} {Boltzmann} {Planar}
  {Maps}}, Electron. J. Combin., 23 (2016), pp.~1--28.

\bibitem{budd_geometry_2017}
{\sc T.~Budd and N.~Curien}, {\em Geometry of infinite planar maps with high
  degrees}, Electron. J. Probab., 22 (2017), pp.~1--37.

\bibitem{burago_course_2001}
{\sc D.~Burago, Y.~Burago, and S.~Ivanov}, {\em A {Course} in {Metric}
  {Geometry}}, Graduate {Studies} in {Mathematics}, American Mathematical
  Society, June 2001.

\bibitem{curien_peeling_2016}
{\sc N.~Curien}, {\em Peeling random planar maps}, 2016.
\newblock Cours Peccot, Collège de France.

\bibitem{curien_percolation_2014}
{\sc N.~Curien and I.~Kortchemski}, {\em Percolation on random triangulations
  and stable looptrees}, Probab. Theory Relat. Fields, 163 (2014), pp.~pp
  303--337.

\bibitem{curien_random_2014}
\leavevmode\vrule height 2pt depth -1.6pt width 23pt, {\em Random stable
  looptrees}, Electron. J. Probab., 19 (2014), pp.~1--35.

\bibitem{denisov_large_2008}
{\sc D.~Denisov, A.~B. Dieker, and V.~Shneer}, {\em Large deviations for random
  walks under subexponentiality: {The} big-jump domain}, Ann. Probab., 36
  (2008), pp.~1946--1991.

\bibitem{doney_exact_1982}
{\sc R.~A. Doney}, {\em On the exact asymptotic behaviour of the distribution
  of ladder epochs}, Stochastic Processes and their Applications, 12 (1982),
  pp.~203--214.

\bibitem{duplantier_liouville_2011}
{\sc B.~Duplantier and S.~Sheffield}, {\em Liouville quantum gravity and
  {KPZ}}, Inventiones Mathematicae, 185 (2011), pp.~333--393.

\bibitem{feller_introduction_1971}
{\sc W.~Feller}, {\em An {Introduction} to {Probability} {Theory} and {Its}
  {Applications}, {Vol}. 2}, Wiley, 2nd~ed., 1971.

\bibitem{gorny_geometry_2017}
{\sc M.~Gorny, E.~Maurel-Segala, and A.~Singh}, {\em The geometry of a critical
  percolation cluster on the {UIPT}}, arXiv:1701.01667 [math],  (2017).

\bibitem{gwynne_convergence_2017}
{\sc E.~Gwynne and J.~Miller}, {\em Convergence of percolation on uniform
  quadrangulations with boundary to {SLE}$_6$ on $\sqrt{8/3}$-{Liouville}
  quantum gravity}, arXiv:1701.05175 [math-ph],  (2017).

\bibitem{ibragimov_independent_1971}
{\sc I.~A. Ibragimov and J.~V. Linnik}, {\em Independent and stationary
  sequences of random variables}, Wolters-{Noordhoff} {Publishing}, Groningen,
  1971.

\bibitem{janson_scaling_2015}
{\sc S.~Janson and S.~O. Stefánsson}, {\em Scaling limits of random planar
  maps with a unique large face}, Ann. Probab., 43 (2015), pp.~1045--1081.

\bibitem{knizhnik_fractal_1988}
{\sc V.~G. Knizhnik, A.~M. Polyakov, and A.~B. Zamolodchikov}, {\em Fractal
  structure of 2d-quantum gravity}, Mod. Phys. Lett. A, 03 (1988),
  pp.~819--826.

\bibitem{kortchemski_invariance_2012}
{\sc I.~Kortchemski}, {\em Invariance principles for {Galton}-{Watson} trees
  conditioned on the number of leaves}, Stochastic Processes and their
  Applications, 122 (2012), pp.~3126--3172.
\newblock arXiv: 1110.2163.

\bibitem{le_gall_uniqueness_2013}
{\sc J.-F. Le~Gall}, {\em Uniqueness and universality of the {Brownian} map},
  Ann. Probab., 41 (2013), pp.~2880--2960.

\bibitem{le_gall_scaling_2011}
{\sc J.-F. Le~Gall and G.~Miermont}, {\em Scaling limits of random trees and
  planar maps}, arXiv:1101.4856 [math],  (2011).

\bibitem{marckert_invariance_2007}
{\sc J.-F. Marckert and G.~Miermont}, {\em Invariance principles for random
  bipartite planar maps}, Ann. Probab., 35 (2007), pp.~1642--1705.

\bibitem{marzouk_scaling_2016}
{\sc C.~Marzouk}, {\em Scaling limits of random bipartite planar maps with a
  prescribed degree sequence}, arXiv:1612.08618 [math],  (2016).

\bibitem{miller_cle_2017}
{\sc J.~Miller, S.~Sheffield, and W.~Werner}, {\em {CLE} percolations}, Forum
  of Mathematics, Pi, 5 (2017).

\bibitem{menard_percolation_2014}
{\sc L.~Ménard and P.~Nolin}, {\em Percolation on uniform infinite planar
  maps}, Electron. J. Probab., 19 (2014).

\bibitem{richier_universal_2015}
{\sc L.~Richier}, {\em Universal aspects of critical percolation on random
  half-planar maps}, Electron. J. Probab., 20 (2015).

\bibitem{richier_incipient_2017}
\leavevmode\vrule height 2pt depth -1.6pt width 23pt, {\em The {Incipient}
  {Infinite} {Cluster} of the {Uniform} {Infinite} {Half}-{Planar}
  {Triangulation}}, arXiv:1704.01561 [math],  (2017).

\bibitem{richier_limits_2017}
\leavevmode\vrule height 2pt depth -1.6pt width 23pt, {\em Limits of the
  boundary of random planar maps}, Probab. Theory Relat. Fields,  (2017),
  pp.~1--39.

\bibitem{watabiki_construction_1995}
{\sc Y.~Watabiki}, {\em Construction of {Non}-critical {String} {Field}
  {Theory} by {Transfer} {Matrix} {Formalism} in {Dynamical} {Triangulation}},
  Nuclear Physics B, 441 (1995), pp.~119--163.

\end{thebibliography}

\end{document}